\documentclass[11pt]{article}
\usepackage{amsfonts}
\usepackage{amssymb}
\usepackage{amsthm}
\usepackage{amsmath}
\usepackage{graphicx}
\usepackage{empheq}
\usepackage{indentfirst}
\usepackage{cite}
\usepackage{mathrsfs}
\usepackage{graphics}
\newtheorem{theorem}{\textbf{Theorem}}[section]
\newtheorem{lemma}{\textbf{Lemma}}[section]
\newtheorem{proposition}{\textbf{Proposition}}[section]
\newtheorem{corollary}{\textbf{Corollary}}[section]
\newtheorem{remark}{\textbf{Remark}}[section]
\newtheorem{definition}{\textbf{Definition}}[section]

\allowdisplaybreaks[4]

\def\be{\begin{equation}}
\def\ee{\end{equation}}
\def\bea{\begin{eqnarray}}
\def\eea{\end{eqnarray}}
\def\bt{\begin{theorem}}
\def\et{\end{theorem}}
\def\bl{\begin{lemma}}
\def\el{\end{lemma}}
\def\br{\begin{remark}}
\def\er{\end{remark}}
\def\bp{\begin{proposition}}
\def\ep{\end{proposition}}
\def\bc{\begin{corollary}}
\def\ec{\end{corollary}}
\def\bd{\begin{definition}}
\def\ed{\end{definition}}

\def\non{\nonumber }

\begin{document}

\title{Well-posedness and longtime behavior \\ for the modified phase-field crystal equation}

\author{
{\sc Maurizio Grasselli}\footnote{Dipartimento di Matematica, Politecnico di Milano, Milano 20133, Italy, \textit{maurizio.grasselli@polimi.it}}
\ and
\ {\sc Hao Wu}\footnote{School of Mathematical Sciences and Shanghai Key Laboratory for Contemporary Applied Mathematics, Fudan University, Shanghai 200433, China, \textit{haowufd@yahoo.com}}
}

\date{\today}

\maketitle


\begin{abstract}
\noindent We consider a modification of the so-called phase-field crystal (PFC) equation introduced by K.R. Elder et al. This variant has recently been proposed by P. Stefanovic et al. to distinguish between elastic relaxation and diffusion time scales. It consists of adding an inertial term (i.e. a second-order time derivative) into the PFC equation. The mathematical analysis of the resulting equation is more challenging with respect to the PFC equation, even at the well-posedness level. Moreover, its solutions do not regularize in finite time as in the case of PFC equation. Here we analyze the modified PFC (MPFC) equation endowed with periodic boundary conditions. We first prove the existence and uniqueness of a solution with initial data in a bounded energy space. This solution satisfies some uniform dissipative estimates which allow us to study the global longtime behavior of the corresponding dynamical system.
In particular, we establish the existence of an exponential attractor. Then we demonstrate that any trajectory originating from the bounded energy phase space does converge to a unique equilibrium. This is done by means of a suitable version of the {\L}ojasiewicz-Simon inequality. A convergence rate estimate is also given.


\medskip\noindent
\textbf{MRS 2010:} 35Q82, 37L99, 74N05, 82C26.
\end{abstract}

\section{Introduction}
\setcounter{equation}{0}
\noindent
The so-called phase-field crystal (PFC) equation has been recently employed to model and simulate the dynamics of crystalline materials, including crystal growth in a supercooled liquid, dendritic and eutectic solidification, epitaxial growth, and so on (see \cite{EG,EKHG}, cf. also \cite{OS08,PDA}). In the phase-field crystal approach,  the number density of atoms is approximated by using a phase function $\phi$. This function tends to minimize the following (dimensionless) free energy functional over a spatial bounded (and usually periodic) domain $Q$
\be
 E(\phi)=\int_{Q} \left(\frac12|\Delta \phi|^2-|\nabla \phi|^2+F(\phi)\right) dx,\label{EE}
\ee
where  $$ F(\phi)=\frac{1-\epsilon}{2}\phi^2+\frac14 \phi^4.$$
The parameter $\epsilon$ is a constant with physical significance proportional to the undercooling, i.e., $\epsilon \sim T_e-T$, $T_e$ being the equilibrium temperature at which the phase transition occurs \cite{EGL11}.

The evolution of $\phi$ is thus governed by a (conserved) gradient flow generated by the Fr\'{e}chet derivative of $E$ (cf. \cite{EKHG,EG}), that is,
\be
\phi_t=\Delta \left\{ \big[-\epsilon+(1+\Delta)^2\big]\phi+\phi^3\right \}\label{pfc}
\ee
in dimensionless units, where the spatial coordinates are
measured in units proportional to the lattice constant.
The PFC equation \eqref{pfc} is a sixth--order partial differential equation preserving the total mass, which can be viewed as the analog of the (fourth--order) Swift--Hohenberg equation (cf. \cite{SH77}).  Equation \eqref{pfc} describes the microstructure of solid-liquid systems at inter-atomic length scales and provides a possibly accurate way to model crystal dynamics, especially defect dynamics in atomic-scale resolution. For more details about the modeling based on the phase-field crystal approach, we refer to the recent review \cite{ELWGTTG12} (see also \cite{EGL11,PDA}).  However, one major disadvantage of equation \eqref{pfc} is that it fails to distinguish between the elastic relaxation and diffusion time scales (see, e.g., \cite{EG, SHP06}). In order to overcome this difficulty and to incorporate both fast elastic relaxation (e.g., in a rapid quasi-phononic time scale) and slower mass diffusion, a modified phase-field crystal (MPFC) model was recently proposed in \cite{SHP06} (cf. also \cite{GDL09,GE11,SHP09} and \cite[Section 3.1.1.2]{ELWGTTG12}):
\be
\beta\phi_{tt}+\phi_t=\Delta \big \{ \big[-\epsilon+(1+\Delta)^2\big]\phi+\phi^3\big \},\label{mpfc}
\ee
where $\beta$ is a (positive) relaxation time. On the other hand, in the context of phase field techniques applied to fast phase transitions, equations like \eqref{mpfc} have been derived in \cite{GJ05} to take large deviations from thermodynamic equilibrium into account. We observe that the presence of the inertial term is not a minor modification from the mathematical viewpoint. For instance, contrary to the parabolic equation \eqref{pfc}, solutions to \eqref{mpfc} do not regularize in finite time.

The modified phase-field crystal (MPFC) equation \eqref{mpfc} has been recently studied from the numerical analysis point of view, in \cite{BHLWWZ,WW11}, while the corresponding literature on the simpler PFC equation \eqref{pfc} is more abundant (cf., e.g., \cite{BRV,CW,EW,GN,WWL09,HWWL09}). In particular, the authors derived different types of unconditionally energy stable finite difference schemes based on a suitable convex splitting for the free energy $E$. On the theoretical side, to the best of our knowledge, the only available results are the existence of a \emph{weak} solution and of a unique \emph{strong} solution to the MPFC equation \eqref{mpfc} up to any positive final time $T>0$. They were both proven in \cite{WW10} by using a time discretization  scheme and taking the initial value of $\phi_t$ equal to zero in order to ensure mass conservation (see Section 2). The authors observe: \emph{Because of the
presence of the second order temporal derivatives, the establishment of a global strong
solution and smooth solution for \eqref{mpfc} is very subtle} (see \cite[Section 2.3]{WW10}). Indeed, the existence issue has been
investigated only partially so far. For instance, no existence result is available for the so-called \emph{energy} solutions (see Definition \ref{soldef} below). Such solutions are natural in the sense that they are only required to have finite free energy and they are more general than the weak ones. Besides, uniqueness has been proven only for strong solutions in \cite{WW10}. Uniqueness of energy (or weak) solutions is not straightforward because of the spatial regularity gap between $\phi$ and $\phi_t$. On the other hand, the analysis of the longtime behavior of solutions to \eqref{mpfc} is a further important issue that has not been explored so far. Due the dissipative nature of \eqref{mpfc}, one would expect to find a global attractor as well as an exponential attractor. Also, the convergence of a trajectory to a single equilibrium turns out to be nontrivial, since the stationary set can have a quite complicated structure (see, e.g., \cite{PR} for an analysis of the stationary one-dimensional equation).

The goal of the present contribution is to establish first the well-posedness of \eqref{mpfc} in the energy space without any restriction
on the initial value of $\phi_t$. A consequence of this is that the mass is conserved only asymptotically and the
corresponding (dissipative) dynamical system is no longer a gradient system (i.e. there is no Lyapunov functional).
Well-posedness is proven by combining standard (e.g. a Galerkin-type approximation scheme) and nonstandard arguments. Then, we demonstrate the existence of global and exponential attractors as well as  the convergence
of single trajectories to single stationary states. Such results are based on the (uniform) precompactness of solutions, which
is obtained through a suitable decomposition of the solution (due to the hyperbolic-like nature of the equation, see Section 5). Besides, the related proofs require some care also because the system does not possess a Lyapunov functional. In order to overcome this obstacle, we shall treat the autonomous system \eqref{mpfc} by using some techniques employed for  non-autonomous evolution equations \cite{GW,H06,HT01} (see Section 6).

For the sake of convenience, we rewrite \eqref{mpfc} into the following form
 \be
\beta \phi_{tt}+\phi_t=\Delta[\Delta^2 \phi+2\Delta \phi+f(\phi)],\qquad \text{ in } \; Q\times (0,+\infty),\label{e1}
 \ee
where
 \be
 f(\phi)=F'(\phi)=\phi^3+(1-\epsilon)\phi.\label{f}
 \ee
 Here, $Q=(0,1)^n$ is a unit cube in $\mathbb{R}^n$, $n\leq 3$, without loss of generality.
 The MPFC equation \eqref{e1} is considered in the periodic setting and is subject to the initial conditions
\be
\phi|_{t=0}=\phi_0(x),\quad \phi_t|_{t=0}=\phi_1(x), \qquad x\in Q.\label{e2}
\ee
We note that the choice of periodic boundary
conditions is realistic since in the crystalline materials the patterns of the nanostructures statistically repeat
throughout the domain, which is much larger than the length-scales of atoms. Nevertheless, our theoretical results also hold for homogeneous Neumann boundary conditions, as well as for mixed periodic-homogeneous Neumann boundary conditions.

The plan of this paper goes as follows.
In Section 2 we introduce a notion of bounded \emph{energy} solution for the MPFC equation \eqref{e1} with \eqref{e2}.
This solution is more general than the \emph{weak} solution considered in \cite{WW10} (see Definition \ref{soldef}). A somewhat similar situation happens in the case of the modified
Cahn--Hilliard equation (see \cite{GSZ,GSZ2,GSSZ}, cf. also \cite{GGMP1,ZM05} for the one-dimensional case). However, the latter equation is technically a bit more challenging. For instance, in the present case, the sixth-order spatial operator ensures that
the weak solutions are globally bounded in three dimensions. This is not true for the modified Cahn-Hilliard equation.
Then, in Section 3, we provide some (formal) a priori global dissipative estimates that will be useful in the sequel (cf. Lemmas \ref{es} and \ref{es1}). The well-posedness is proven in Section 4 so that we can define a dissipative semigroup
acting on the energy phase space (cf. Theorems \ref{exe} and \ref{abs}).
The existence of the global attractor (Theorem \ref{gloatt}) as well as of an exponential attractors (Theorem \ref{exat}) are established in Section 5.
The final Section 6 is devoted to prove the convergence of energy solutions
to single equilibria (Theorem \ref{convergence}) by using a suitable adaptation of the {\L}ojasiewicz--Simon inequality (cf. Lemma \ref{ls}).

We conclude by observing that an interesting open issue is the construction of a robust family of exponential attractors with respect to the relaxation time $\beta$ (see \cite[Section 3.3]{MZ}). More precisely, the goal is to show that the existence of
a family of exponential attractors depending on $\beta\geq 0$, which is (H\"{o}lder) continuous with respect to $\beta$.
Such a result essentially says that the non-transient dynamics of the MPFC equation \eqref{mpfc} is \emph{close} to the one
of the PFC equation \eqref{pfc} in a quantitative way. The result can be obtained by using, e.g., the argument devised in
\cite{MPZ} for the damped semilinear wave equation. Nonetheless, in the present case, there is an additional difficulty related to the already mentioned high (spatial) regularity gap between $\phi$ and
$\phi_t$. This seems to require the construction of a more regular invariant set with respect to what is needed here (see Proposition \ref{abs2}) in order to estimate the energy norm of the difference between the solution to the MPCF and the PFC equations, respectively. The corresponding analysis will be carried out in a forthcoming paper.


\section{Preliminaries}
\setcounter{equation}{0}
\textit{Notations and functional spaces.}
Let $V$ be a real Banach space with norm $\|\cdot\|_V$ and denote $V^*$ its dual space. By $<\cdot,\cdot>_{V^*,V}$ we indicate the duality product between $V$ and $V^*$. We denote by $L(X;Y)$ the space of all bounded linear operators
from a Banach space $X$ into a second Banach space $Y$, and we simply write $L(X)=L(X;X)$. Next, we denote by $H^m_p(Q)$, $m\in \mathbb{N}$, the space of functions that are in $H^m_{loc}(\mathbb{R}^n)$ and periodic with the period $Q$. For an arbitrary $m\in \mathbb{N}$, $H^m_p(Q)$ is a Hilbert space for the scalar product $(u,v)_{m}=\sum_{|\kappa|\leq m}\int_Q D^\kappa u(x)D^\kappa v(x) dx$ ($\kappa$ being a multi-index) and its associated norm $\|u\|_m=\sqrt{(u,u)_m}$.  For $m=0$, $H^0_p(Q)=L^2_p(Q)$ and the inner product as well as the norm on $L^2_p(Q)$ are simply indicated by $(\cdot, \cdot)$ and $\|\cdot\|$, respectively.

 The mean value of any function $u\in L^2_p(Q)$ is denoted by $\langle u\rangle=|Q|^{-1}\int_{Q} u dx$ and we set $\overline u=u-\langle u\rangle$. For the sake of simplicity, we assume $|Q|=1$ hereafter.
The dual space of $H^{m}_p(Q)$ is denoted by $H^{-m}_p(Q)$, which is equipped
with the operator norm
 given by
 $\|\mathcal{T}\|_{-m}=\sup_{\|u\|_m=1,\ u\in H^m_p(Q)}|\mathcal{T}(u)|$.
 For $m=1$, we introduce an equivalent and more convenient norm associated with the inner product
$$(u, v)_{-1}= (\nabla \psi_u, \nabla \psi_v)+\langle u\rangle\langle v\rangle,\quad \forall\, u, v\in H^{-1}_p(Q),$$
where $\psi_u$ (respectively $\psi_v$) is the unique solution to the elliptic equation in $Q$ subject to periodic boundary conditions:
$$ -\Delta \psi_u=u-\langle u\rangle, \quad \text{with}\ \langle \psi_u\rangle=0.$$
We denote by $\dot{H}^m_p(Q)=\{u\in H^m_p(Q):\ \langle u\rangle=0\}$ the Sobolev spaces for functions with zero mean.
For any $u, v\in \dot{L}^2_p(Q)$, we have
$$(u, v)_{-1}=(\nabla \psi_u, \nabla \psi_v)\quad \text{and}\quad \|u\|_{-1}=\|\nabla \psi_u\|.$$
Then we observe that $$A_0=-\Delta: \dot{H}_p^2\mapsto \dot{L}^2_p(Q)$$ is a positive operator so that its powers $A_0^s$ ($s\in\mathbb{R}$) are well defined. In particular, for $s=-1$,
$$ (u, v)_{-1}=(A_0^{-1} u, v)=(u, A^{-1}_0 v)=(A_0^{-\frac12} u, A_0^{-\frac12} v) \quad \text{and}\quad \|u\|_{-1}=\|A_0^{-\frac12} u\|.$$
We also need to introduce the product spaces
\bea
&& \mathbb{X}_0=H^2_p(Q)\times H^{-1}_p(Q),\quad \mathbb{X}_1=H^3_p(Q)\times L^2_p(Q),\non\\
&& \mathbb{X}_2=H^4_p(Q)\times H^1_p(Q),\quad\ \ \mathbb{X}_3=H^6_p(Q)\times H^3_p(Q)\non
\eea
 endowed with the graph norm.

\textit{Mass conservation}. We recall that an important feature of the (parabolic)  PFC equation \eqref{pfc} is that it enjoys the mass conservation property, namely, $$\langle\phi(t)\rangle=\langle\phi_0\rangle,\quad \forall\, t\geq 0.$$
However, the mass conservation may fail for the (hyperbolic) MPFC equation \eqref{e1}. Nevertheless, it still obeys a conservation law of different type. To this end,
 we (formally) integrate \eqref{e1} over $Q$ and using the periodic boundary condition, we find
 \be
 \beta\langle \phi_{tt}\rangle+\langle\phi_t \rangle=0,\label{conODE1}
 \ee
 which yields the following conservative property after integrating it with respect to time:
 \bl \label{conser} Let $(\phi, \phi_t)$ be a regular solution to problem \eqref{e1}--\eqref{e2}.
 Then we have
 \be
 \beta\langle\phi_t(t)\rangle+\langle\phi(t)\rangle=\beta\langle\phi_1\rangle+\langle\phi_0\rangle, \quad \forall\, t\geq 0.\label{conODE2}
 \ee
 \el

 For the sake of simplicity, we denote
 \be
 M:=\beta\langle\phi_1\rangle+\langle\phi_0\rangle\quad \text{and} \quad \mathcal{A}(t):=\langle\phi_1\rangle e^{-\frac{t}{\beta}}.\non
 \ee
  By solving the ODE system \eqref{conODE1}--\eqref{conODE2} for ($\langle\phi(t)\rangle, \langle\phi_t(t)\rangle$) with initial conditions $\langle\phi(0)\rangle=\langle\phi_0\rangle$ and $\langle\phi_t(0)\rangle=\langle\phi_1\rangle$, we have the following explicit expressions of $\langle\phi\rangle$ and $\langle\phi_t\rangle$:
 \bea
 \langle\phi_t(t)\rangle&=&\mathcal{A}(t),\quad \forall\, t\geq 0, \label{mde1}\\
 \langle\phi(t)\rangle&=&M-\beta \mathcal{A}(t),\quad \forall\, t\geq 0.\label{mde2}
 \eea
 \br
 It is easy to see that if $(\phi, \phi_t)$ is a solution to problem \eqref{e1}--\eqref{e2}   with initial data $\phi_1$ satisfying the zero-mean assumption $\langle\phi_1\rangle=0$, then $\langle \phi_t(t)\rangle=0$ and, in particular, the mass conservation  $\langle \phi(t)\rangle=\langle\phi_0\rangle$ holds  for all $t\geq 0$ (cf. \cite{WW10, WW11}).
 \er

 \textit{Energy dissipation}. Another important property of the PFC equation \eqref{pfc} is that its total energy $E(\phi)$ (cf. \eqref{EE}) is decreasing with respect to time, namely,
 \be
 \frac{d}{dt}E(\phi)=-\|\phi_t\|_{-1}^2\leq 0.\non
 \ee
 As far as the MPFC equation \eqref{e1} is concerned, in the special case such that $\langle\phi_1\rangle=0$, one can introduce the following ``pseudo energy" to problem \eqref{e1}--\eqref{e2}   (cf. \cite{WW10})
 \be
 \widetilde{\mathcal{E}}(t)=\frac{\beta}{2}\|\phi_t(t)\|_{-1}^2+ E(\phi(t)),\label{pseE}
 \ee
 which is also nonincreasing in time, i.e.,
 \be
 \frac{d}{dt} \widetilde{\mathcal{E}}(t)=-\|\phi_t\|_{-1}^2\leq 0.\label{BEL}
 \ee
 We note that the theoretical results in \cite{WW10} were only obtained under the specific assumption $\phi_1=0$ that yields $\langle\phi_1\rangle=0$.

  Hereafter, we shall \emph{not} imposed any restriction on the mean of $\phi_1$. Therefore, the dissipative property \eqref{BEL} no longer holds if $\langle\phi_1\rangle\neq 0$.  In this more general situation, problem \eqref{e1}--\eqref{e2} does not have a Lyapunov function. In order to study the long-time behavior of the system, we need to introduce a modified ``pseudo energy" to problem \eqref{e1}--\eqref{e2} instead of \eqref{pseE}, i.e.,
 \be
 \mathcal{E}(t)=\frac{\beta}{2}\|\overline{\phi_t}(t)\|_{-1}^2+ E(\phi(t)).\label{pseE1}
 \ee
 It is easy to check that the energy functional $\mathcal{E}(t)$ coincides with $\widetilde{\mathcal{E}}(t)$ (cf. \eqref{pseE}) provided that $\langle \phi_t\rangle=0$.

\textit{Notion of solution}. In order to distinguish the solutions according to their regularities, we introduce
the following terminology (reminiscent of \cite{GSSZ}, in which the Cahn-Hilliard equation with inertial term was considered).

\begin{definition}
\label{soldef}
Let $T>0$ be given (possibly $T=+\infty$).

 (1) A pair $(\phi,\phi_t)$ is called an \emph{energy solution} to problem \eqref{e1}--\eqref{e2}  , if
 \be
 (\phi,\phi_t)\in L^\infty (0,T; \mathbb{X}_0),\ \ \phi_{tt}\in L^\infty(0, T; H^{-4}_p(Q)),\ \ \mathcal{E}\in L^\infty(0,T)\label{sreg}
 \ee
  and the following relations hold
\bea&&
A_0^{-1}(\beta \phi_{tt}+\phi_t)+ \Delta ^2 \phi +2\Delta \phi+f(\phi)-\langle f(\phi)\rangle
=0,\non\\
 && \qquad \qquad \qquad \qquad\qquad \mbox{in}\  D(A_0^{-1}), \quad \text{a.e. in}\ (0,T), \label{e1ae}\\
&&
\phi|_{t=0}=\phi_0 \ \mbox{in} \ H^2_p(Q),\quad \phi_t|_{t=0}=\phi_1 \ \mbox{in} \ H^{-1}_p(Q).\label{e2ae}
\eea

(2) We say that an energy
solution to problem \eqref{e1}--\eqref{e2}   is a \emph{weak solution}, if $(\phi,\phi_t)\in L^\infty (0,T; \mathbb{X}_1)$ and $\phi_{tt}\in L^\infty(0, T; H^{-3}_p(Q))$. The relation \eqref{e1ae} and initial conditions \eqref{e2ae} can correspondingly be interpreted in a stronger sense.

(3) We say that an energy
solution to problem \eqref{e1}--\eqref{e2}   is a \emph{strong solution}, if $(\phi,\phi_t)\in L^\infty (0,T; \mathbb{X}_3)$ and $\phi_{tt}\in L^\infty(0, T; L^2_p(Q))$. In this case, $\phi$ satisfies the
equation \eqref{e1} almost everywhere in $Q\times(0,T)$.
\end{definition}

\br
We remark that the initial conditions in \eqref{e2ae} make sense, since the regularity of $(\phi, \phi_t)$ in \eqref{sreg} ensures the weak continuity $ (\phi,\phi_t)\in C_w(0,T; \mathbb{X}_0)$. Here, we denote $C_w(0,T; X)$ ($X$ being a Banach space) as the topological vector space of
all weakly continuous functions $f : [0, T )\to X$.
\er

\br
       For the sake of simplicity, in this paper we only treat the nonlinearities $f$ of the physically relevant form \eqref{f}. Moreover, in the subsequent analysis, we do not have to impose the restriction $\alpha:=1-\epsilon  >0$ as in \cite{WW10, WW11, WWL09}.
       Actually, our results hold for more general (possibly non-convex) nonlinearities.
       For instance, we can take $f$ satisfying the following assumptions:
\begin{itemize}
 \item[(H1)] $ f\in C^{2,1}_{loc}(\mathbb{R};\mathbb{R})$, $f(0)=0$,
 \item[(H2)]  $ \liminf_{|s|\to+\infty} f'(s)>0,$
 \item[(H3)]  $\liminf_{|s|\to+\infty}\frac{f(s)}{s}=+\infty$.
 \end{itemize}
 Only in the final section we shall further require $f$ to be real analytic in order to prove the convergence to single equilibria. Note that the physically relevant  nonlinearity  $f(y)= y^3+(1-\epsilon)y$ with $\epsilon\in \mathbb{R}$ also satisfies (H1)--(H3).
 \er

  %

 \section{A priori dissipative estimates}
\setcounter{equation}{0}
\noindent
In this section, we first derive some a priori dissipative estimates for the solutions to problem \eqref{e1}--\eqref{e2}. Such estimates will be crucial in the subsequent sections. The following calculations are performed in a formal way. However, they can be justified by working within a suitable Faedo--Galerkin approximation scheme (cf. Section 4 below) and then passing to the limit.

From now on, the symbols $c$ and $c_i$ ($i \in \mathbb{N}$) will denote positive constants
depending on $f, \beta, \epsilon, |Q|$, but independent of the initial datum
and of time. Their values are allowed to vary even within the same line. Analogously, $\mathcal{Q}_i: \mathbb{R}\to\mathbb{R}$  stand for generic nonnegative monotone functions. Capital letters
like $C$ or $C_i$ will be used to indicate constants that have other dependencies (in most
cases, on the initial datum). The symbol $c_Q$ will denote some embedding
constants depending only on the domain $Q$.

\bl\label{es}
Suppose $(\phi, \phi_t)$ is a regular solution to problem \eqref{e1}--\eqref{e2}. Then the following dissipative estimate holds:
 \bea
 && \|\phi(t)\|_{2}^2+ \|\overline{ \phi_t}(t)\|_{-1}^2 + \int_0^t e^{-\rho_1(t-s)}(\|\overline{\phi_t}(s)\|_{-1}^2+\| \phi(s)\|_2^2)ds \non\\
 & \leq& \mathcal{Q}_1(\|\phi_0\|_{2}, \|\overline{ \phi_1}\|_{-1}) e^{-\rho_1 t}+\rho_2, \quad \forall\, t\geq 0,\label{disa1}
 \eea
 where the positive constants $\rho_1, \rho_2$ may depend on $\beta$,  $\langle\phi_1\rangle $, $M$ and $|Q|$, but they are independent of $\|\phi_0\|_{2}$, $\|\overline{ \phi_1}\|_{-1}  $ and time $t$.
\el
\begin{proof}
We first rewrite \eqref{e1} into the following form
\be
\beta \phi_{tt}+\phi_t=-A_0 (\Delta ^2 \phi +2\Delta \phi+f(\phi)-\langle f(\phi)\rangle). \label{e1a}
\ee
Testing \eqref{e1a} by $A_0^{-1} \overline{\phi_t}$ and $A_0^{-1} \overline{\phi}$, respectively, we get
\be
\frac{d}{dt}\left(\frac{\beta}{2}\|\overline{\phi_t}\|_{-1}^2+ E(\phi)\right)+\|\overline{\phi_t}\|_{-1}^2= \mathcal{A}(t) \int_{Q} f(\phi) dx \label{d1}
\ee
and
\be
\frac{d}{dt}\left(\beta (\overline{\phi_t}, \overline\phi)_{-1}+\frac12\|\overline \phi\|_{-1}^2\right)-\beta \|\overline{\phi_t}\|^2_{-1}+\|\Delta \phi\|^2+\int_{Q} f(\phi) \overline \phi dx = 2\|\nabla \phi\|^2.\label{d2}
\ee
Multiplying \eqref{d2} by a small constant $\eta>0$ (to be determined later) and add the resulting equation to \eqref{d1}, we obtain that
\be
\frac{d}{dt}\mathcal{Y}_1(t)+\mathcal{D}_1(t)\leq \mathcal{R}_1(t),\label{d3}
\ee
where
\bea
\mathcal{Y}_1(t)&=& \frac{\beta}{2}\|\overline{\phi_t}\|_{-1}^2+ E(\phi)+  \eta \beta (\overline{\phi_t}, \overline \phi)_{-1}+\frac{\eta}{2}\|\overline \phi\|_{-1}^2,\nonumber\\
\mathcal{D}_1(t)&=& (1-\eta\beta)\|\overline{\phi_t}\|_{-1}^2+\eta \|\Delta \phi\|^2+\eta \int_{Q} f(\phi) (\phi-M) dx,\non\\
\mathcal{R}_1(t)&=& (1-\eta\beta)\mathcal{A}(t)\int_{Q} f(\phi) dx+2\eta \|\nabla \phi\|^2.\non
\eea
Using integration by parts and the Cauchy--Schwarz inequality, we get
\be
\|\nabla \phi\|^2 \leq \|\Delta \phi\|\|\phi\|\leq \frac14\|\Delta \phi\|^2+ \|\phi\|^2.\label{intpo}
\ee
Then by the expression \eqref{f}, the Young inequality and the Sobolev embedding $H^2_p(Q)\hookrightarrow L^\infty(Q)$ $(n\leq 3$), we have
\be
C(\|\phi\|_{2})\geq E(\phi)\geq \frac14\|\Delta \phi\|^2+ \|\phi\|^2-c_1.\label{esE}
\ee
It follows from the Cauchy--Schwarz inequality that
\be
\eta \beta |(\overline{\phi_t}, \overline\phi)_{-1}|\leq \frac{\beta}{4}\|\overline{\phi_t}\|_{-1}^2+\eta^2\beta\|\overline \phi\|_{-1}^2. \label{cs1}
\ee
As a result, for $\eta\in (0, \frac{1}{2\beta})$, we obtain
\be
C(\|\overline{ \phi_t}\|_{-1}, \|\phi\|_{H^2})\geq \mathcal{Y}_1(t)\geq \frac{\beta}{4}\|\overline{\phi_t}\|_{-1}^2+\frac14\|\Delta \phi\|^2+ \|\phi\|^2+\frac{\eta}{4}\|\overline\phi \|_{-1}^2-c_1.\label{Y1}
\ee
Next, it follows that
\bea
\int_{Q} f(\phi) (\phi-M) dx &\geq& -c_2\|\phi-M\|^2 +c_3\int_{Q} F(\phi) dx-c_4\non\\
&\geq & -2c_2\|\phi\|^2 +c_3\int_{Q} F(\phi) dx-(c_4+2M^2).
\eea
Thus, for $\eta\in (0, \frac{1}{2\beta})$, we deduce
\bea
\mathcal{D}_1(t)&\geq& \frac12  \|\overline{\phi_t}\|_{-1}^2+\eta\|\Delta \phi\|^2+\eta c_3\int_{Q} F(\phi)dx -2\eta c_2 \|\phi\|^2-\eta(c_4+2M^2).
\eea
By the definition of $\mathcal{Y}_1(t)$ and \eqref{cs1}, for certain small $c_5>0$, we get
\be
\frac12 \mathcal{D}_1(t) \geq  c_5\mathcal{Y}_1(t)-c_6\|\overline \phi\|_{-1}^2   -2\eta c_2 \|\phi\|^2-\eta(c_4+2M^2).
\ee
Next, for $\kappa>0$ satisfying $4\kappa |\langle\phi_1\rangle|\leq \eta c_3$, using the Young inequality and \eqref{intpo}, we can see that
\bea
\mathcal{R}_1(t)&\leq& |\mathcal{A}(t)|\int_{Q} |f(\phi)| dx+2\eta \|\nabla \phi\|^2\non\\
&\leq& \frac{\eta c_3}{2}\int_{Q} F(\phi)dx + c_\kappa |\langle\phi_1\rangle|+ \frac{\eta}{2}\|\Delta \phi\|^2+ 2\eta \|\phi\|^2,\non
\eea
where $c_\kappa$ depends on $\kappa$ and $Q$.
Collecting the above estimates together,
we infer from inequality \eqref{d3} that
\be
\frac{d}{dt}\mathcal{Y}_1(t)+2c_7(\mathcal{Y}_1(t)+\|\overline{\phi_t}\|_{-1}^2+\| \phi\|_2^2)\leq c_8 \|\phi\|^2+ c_9,\label{d3a}
\ee
where $c_7, c_8, c_9$ may depend on $\beta, \eta, Q, M$ and $|\langle\phi_1\rangle|$. Then, by the Young inequality and \eqref{esE}, we get
\be
\frac{d}{dt}\mathcal{Y}_1(t)+c_7(\mathcal{Y}_1(t)+\|\overline{\phi_t}\|_{-1}^2+\| \phi\|_2^2)\leq c_{10},\label{d3b}
\ee
which yields that
\be
\mathcal{Y}_1(t)+\int_0^t e^{-c_7(t-s)}(\|\overline{\phi_t}(s)\|_{-1}^2+\| \phi(s)\|_2^2)ds \leq \mathcal{Y}_1(0)e^{-c_7 t}+\frac{c_{10}}{c_7},\quad  \forall \, t\geq 0.
\ee
Thus, we can deduce our conclusion from \eqref{Y1}. The proof is complete.
\end{proof}

If the initial data are more regular, higher-order dissipative estimates can be obtained, namely, we have

\bl\label{es1}
Suppose $(\phi,\phi_t)$ is a regular solution to problem \eqref{e1}--\eqref{e2}. Then the following estimate holds:
 \bea
 && \|\phi(t)\|_{3}^2+ \|\overline{\phi_t}(t)\|^2 + \int_0^t e^{-\rho_3(t-s)}(\| \phi(s)\|_3^2+\|\overline{\phi_t}(s)\|^2)ds \non\\
 & \leq& \mathcal{Q}(\|\phi_0\|_{3}, \|\overline{ \phi_1}\|) e^{-\rho_3 t}+\rho_4, \quad \forall\, t\geq 0,\label{disa1a}
 \eea
 where the positive constants $\rho_3, \rho_4$  may depend on $\beta$,  $\langle\phi_1\rangle $, $M$ and $|Q|$, but independent of $\|\phi_0\|_{3}$, $\|\overline{ \phi_1}\|$ and time $t$.
\el
\begin{proof}
Testing \eqref{e1a} by $ \overline{\phi_t}$, $\overline{\phi}$, respectively, we obtain
 \bea
 && \frac{d}{dt}\left( \frac{\beta}{2}\|\overline{\phi_t}\|^2+\frac12\|\nabla \Delta \phi\|^2-\|\Delta \phi\|^2+\frac12\int_{Q} f'(\phi)|\nabla \phi|^2 dx\right)+\|\overline{\phi_t}\|^2\non\\
 &=& \frac12 \int_{Q}f''(\phi)|\nabla \phi|^2\overline{\phi_t} dx+\frac12\mathcal{A}(t)\int_{Q}f''(\phi)|\nabla \phi|^2dx\label{d1h}
  \eea
  and
  \be \frac{d}{dt}\left(\beta\int_{Q} \overline{\phi_t}\ \overline{\phi}dx +\frac12\|\overline{\phi}\|^2\right)+\|\nabla \Delta \phi\|^2
  = \beta\|\overline{\phi_t}\|^2+2\|\Delta \phi\|^2+\int_{Q}f(\phi)\Delta \phi dx.\label{d2h}
  \ee
  Multiplying \eqref{d2h} by $\eta'\in (0, \frac{1}{2\beta})$  and adding the resulting equation to \eqref{d1h}, we obtain that
\be
\frac{d}{dt}\mathcal{Y}_2(t)+\mathcal{D}_2(t)\leq \mathcal{R}_2(t),\label{d3h}
\ee
with
\bea
\mathcal{Y}_2(t)&=&\frac{\beta}{2}\|\overline{\phi_t}\|^2+\frac12\|\nabla \Delta \phi\|^2-\|\Delta \phi\|^2+\frac12\int_{Q} f'(\phi)|\nabla \phi|^2 dx\non\\
&& +\beta\eta'\int_{Q} \overline{\phi_t}\ \overline{\phi}dx +\frac{\eta'}{2}\|\overline{\phi}\|^2,\\
\mathcal{D}_2(t)&=& (1-\eta'\beta)\|\overline{\phi_t}\|^2+\eta'\|\nabla \Delta \phi\|^2,\\
\mathcal{R}_2(t)&=& \frac12 \int_{Q}f''(\phi)|\nabla \phi|^2 \overline{\phi_t} dx+\frac12\mathcal{A}(t)\int_{Q}f''(\phi)|\nabla \phi|^2dx\non\\
&&+2\eta'\|\Delta \phi\|^2+\eta'\int_{Q}f(\phi)\Delta \phi dx.
\eea
The uniform dissipative estimate \eqref{disa1} together with the Sobolev embedding $H^2(Q)\hookrightarrow L^\infty(Q)$ ($n\leq 3$) yields the (uniform) global boundedness of $\phi$, that is,
  $$
  \|\phi\|^2_{L^\infty}\leq c\mathcal{Q}_1(\|\phi_0\|_{2}, \|\overline{ \phi_1}\|_{-1}) e^{-\rho_1 t}+c\rho_2.
  $$
As a consequence, thanks to \eqref{f} and the Young inequality, we have
\be
\mathcal{Q}_2(\|\phi\|_{3}, \|\overline{ \phi_t}\|)\geq \mathcal{Y}_2(t)\geq \frac14(\beta\|\overline{\phi_t}\|^2+\|\phi\|_3^2)-c\|\phi\|_2^2
\ee
     Then, using the Sobolev embeddings once more, we are able to estimate the terms in $\mathcal{R}_2$:
    \bea
     \frac12\int_{Q}f''(\phi)|\nabla \phi|^2\overline{\phi_t} dx
    &\leq& \frac18\|\overline{\phi_t}\|^2+\|f''(\phi)\|_{L^\infty}^2\|\nabla \phi\|_{L^4}^4\non\\
    &\leq& \frac18\|\overline{\phi_t}\|^2+\mathcal{Q}_3(\|\phi\|_{L^\infty})\|\phi\|_2^4,
    \eea
    \bea
    \frac12\mathcal{A}(t)\int_{Q}f''(\phi)|\nabla \phi|^2dx&\leq& \frac12\langle\phi_1\rangle e^{-\frac{t}{\beta}}\|f''(\phi)\|_{L^\infty}\|\nabla \phi\|_{L^4}^2\non\\
    &\leq& \mathcal{Q}_3(\|\phi\|_{L^\infty})\|\phi\|_2^4+ \frac{1}{16}|\langle\phi_1\rangle|^2 e^{-\frac{2t}{\beta}},
    \eea
    \be
    2\eta'\|\Delta \phi\|^2+\eta'\int_{Q}f(\phi)\Delta \phi dx\leq c\eta'(\|\phi\|_2^2+\|f(\phi)\|^2)\leq \mathcal{Q}_4(\|\phi\|_2).
    \ee
    Thus we deduce from the expressions of $\mathcal{Y}_2$, $\mathcal{D}_2$ and $\mathcal{R}_2$ that the following inequality holds
    \be
    \frac{d}{dt}\mathcal{Y}_2+c_{11}\mathcal{Y}_2(t)\leq \mathcal{Q}_5(\|\phi\|_2)+\frac{1}{16}|\langle\phi_1\rangle|^2 e^{-\frac{2t}{\beta}},
    \ee
    where $\mathcal{Q}_5$ is a continuous monotone function satisfying $\mathcal{Q}_5(0)=0$. Applying the Gronwall inequality and the dissipative estimate \eqref{es}, we conclude that
    \bea
    \mathcal{Y}_2(t)& \leq &
    \mathcal{Y}_2(0)e^{-c_{11}t}+\int_0^te^{-c_{11}(t-s)}\left(\mathcal{Q}_5(\|\phi(s)\|_2)+\frac{1}{16}|\langle\phi_1\rangle|^2 e^{-\frac{2s}{\beta}}\right)ds\non\\
    &\leq& C(\|\phi_0\|_3, \|\overline{\phi_1}\|) e^{-c_{12}t} +  c_{13},\non
    \eea
which combined with \eqref{es} easily yields estimate \eqref{disa1a}. The proof is complete.
\end{proof}

\section{Existence and uniqueness of energy solution}\setcounter{equation}{0}
\noindent
In this section, we shall establish the existence and uniqueness of energy solutions to problem \eqref{e1}--\eqref{e2}. Namely, we prove
\begin{theorem}\label{exe}
  For any initial data $(\phi_0, \phi_1)\in \mathbb{X}_0$, problem \eqref{e1}--\eqref{e2}   admits a unique global energy solution $(\phi, \phi_t)$. Moreover, any energy solution satisfies the strong time continuity property
\be
\phi\in C^2([0,T];H^{-4}_p(Q))\cap C^1([0,T]; H^{-1}_p(Q))\cap C([0,T]; H^2_p(Q)), \label{regener}
\ee
as well as the following energy identity, for all $ s, \,t\in [0,T]$ with $s<t$,
\be
\mathcal{E}(t)=\mathcal{E}(s)-\int_s^t \|\overline{\phi_t}(\tau)\|_{-1}^2 d\tau+\int_s^t\mathcal{A}(\tau)\int_Q f(\phi(\tau))dx d\tau. \label{enereq}
\ee
 \end{theorem}
Before giving the proof of Theorem \ref{exe}, we first establish an auxiliary result which will be useful to establish the time continuity \eqref{regener} and the energy identity \eqref{enereq} for energy solutions.

Let us consider the following linear equation
\be
\beta \psi_{tt} + \psi_t - \Delta^3 \psi - 2\Delta^2 \psi + \Lambda \overline \psi = G,  \qquad \text{ in } \; Q\times (0,T),\label{lineq1}
\ee
subject to the periodic boundary conditions and the initial conditions
\be
\psi|_{t=0}=\psi_0(x),\quad \psi_t|_{t=0}=\psi_1(x), \quad x\in Q.\label{linei}
\ee
Here, $\Lambda$ is a sufficiently large positive constant and $G$ is a given function satisfying
\be
G \in C([0,T]; \dot{H}^{-1}_p(Q)). \label{lineq2}
\ee
We have
\begin{lemma}\label{lin} For any $(\psi_0,\psi_1)\in \mathbb{X}_0$ and any function $G$ satisfying \eqref{lineq2}, there exists a unique global solution $\psi \in C([0,T];H^2_p(Q)) \cap C^1([0,T];H^{-1}_p(Q))\cap C^2([0,T]; H^{-4}_p(Q))$ to the linear problem \eqref{lineq1}--\eqref{linei}. Moreover, the following energy identity holds
\bea
{\cal E}_0(\psi(t),\psi_t(t)) &=& {\cal E}_0(\psi_0,\psi_1) - \int_0^t \Vert \overline{ \psi_t}(\tau)\Vert_{-1}^2 d\tau\cr \non\\
&& + \int_0^t \int_Q A_0^{-\frac12} G(\tau) A^{-\frac12 }_0 \overline{\psi_t}(\tau) dx d\tau, \quad \forall\, t\in [0,T], \label{enlin2}
\eea
where
\be
{\cal E}_0(\psi,\psi_t)= \frac{\beta}{2}\Vert \overline{\psi_t}\Vert_{-1}^2 +{1\over 2} \| \Delta\psi\|^2 - \| \nabla\psi\|^2 + {\Lambda\over 2}\| \overline{\psi} \|_{-1}^2.
 \label{enlin1}
\ee
\end{lemma}

\begin{proof}
The proof of existence can be done through a density argument combined with the semigroup theory. We first approximate $G$ and the initial data with the following sequences
$$ \{G^m\}\subset C([0,T]; \dot H^3_p(Q)), \quad  G^m\to G\quad \text{strongly in}\ \ C([0,T]; \dot{H}^{-1}_p(Q)),$$
$$ \{(\psi_{0}^m, \psi_{1}^m)\}\subset \mathbb{X}_1, \quad (\psi_{0}^m, \psi_{1}^m)\to (\psi_0,\psi_1)\quad \text{strongly in} \ \ \mathbb{X}_0.$$
Let us denote by $\Psi^m=(\psi^m, \beta \psi_{t}^m)$ a possible solution to the following linear Cauchy problem
\be
\begin{cases}
&\displaystyle{\frac{d}{dt}}\Psi^m+\mathcal{L}\Psi^m=\mathcal{G}^m,\\
& \Psi^m=(\psi_{0}^m, \psi_{1}^m)\in \mathbb{X}_1,
\end{cases}\label{linaa}
\ee
with
\be
\mathcal{L}=\left(
\begin{array}{cc}
0&-\beta^{-1} I\\
-\Delta^3-\Delta^2+\Lambda I & \beta^{-1} I
\end{array}
\right),\quad \mathcal{G}^m=\left(\begin{array}{c}
0\\
\Lambda\langle\psi^m\rangle+ G^m
\end{array}
\right).\non
\ee
For sufficiently large $\Lambda>0$, it is easy to verify that $\mathcal{L}$ is a m-accretive operator on $\mathbb{X}_1$ with $D(\mathcal{L})=\mathbb{X}_3$. As a consequence, the linear problem \eqref{linaa} admits a unique (global) strong solution $\Psi^m\in C([0,T]; \mathbb{X}_3)\cap C^1([0,T];\mathbb{X}_1)$. Let us now write \eqref{lineq1} for the couple of
indices $m$ and $m'$. Then we take the difference, setting $\zeta=\psi^m-\psi^{m'}$, $\zeta_0=\psi_{0}^m-\psi_0^{m'}$, $\zeta_1=\psi_{1}^m-\psi_1^{m'}$, and we test the difference equation by $A_0^{-1} \overline{\zeta_t}$. Integrating the resulting equation in time over $(0,t)$ for $0\leq t \leq T$, we get
\be
\mathcal{E}_0(\zeta, \zeta_t)+\int_0^t \|\overline{\zeta_t}\|_{-1}^2 d\tau= \mathcal{E}_0(\zeta_0, \zeta_1)+\int_0^t A_0^{-\frac12} (G^m-G^{m'}) A_0^{-\frac12} \overline{\zeta_t} d\tau.\label{difee}
\ee
We infer from the definition of $\mathcal{E}_0$ that for sufficiently large $\Lambda>0$, there exist $c, c'>0$ such that
\be
c(\|\psi\|_2^2+\|\psi_t\|_{-1}^2)\geq {\cal E}_0(\psi,\psi_t)\geq  c'\|\overline{\psi}\|_2^2 +\frac{\beta}{2}\Vert \overline{\psi_t}(t)\Vert_{-1}^2.\label{EE0}
\ee
As a consequence, it follows from \eqref{difee} that
\bea
 && \Big(c'\|\overline {\zeta}(t) \|_2^2+\frac{\beta}{2}\|\overline{\zeta_t}(t)\|_{-1}^2\Big)+ \frac12 \int_0^t \|\overline{\zeta_t}(\tau)\|_{-1}^2 d\tau \non\\
 &\leq& c (\|\zeta_0\|_2^2+\|\zeta_1\|_{-1}^2)+ \frac12 \int_0^t \|G^m(\tau)-G^{m'}(\tau)\|_{-1}^2 d\tau.\label{mde0a}
\eea
Moreover, since $\langle G^m\rangle=0$ ($m\in \mathbb{N}$) by assumption, recalling \eqref{mde1} and \eqref{mde2}, we easily see that for $t\in[0,T]$
 \bea
 \langle\zeta_t(t)\rangle&=&\langle\zeta_1\rangle e^{-\frac{t}{\beta}}, \label{mde1a}\\
 \langle\zeta(t)\rangle&=&\beta\langle \zeta_1\rangle +\langle \zeta_0\rangle-\beta \langle\zeta_1\rangle e^{-\frac{t}{\beta}}.\label{mde2a}
 \eea
Thus, taking the supremum with respect to $t\in [0,T]$ in \eqref{mde0a}--\eqref{mde2a}, we deduce that $ \{\psi^m\}$ is  a Cauchy
sequence with respect to the norm of $C([0,T];H^2_p(Q)) \cap C^1([0,T];H^{-1}_p(Q))$. Note that at this stage we only need $G \in L^2(0,T; \dot{H}^{-1}_p(Q))$ instead of \eqref{lineq2}. Besides, the convergence of $\psi_{tt}^m$ can
be proved by comparison in \eqref{lineq1}. Therefore, $ \{\psi^m\}$  strongly converges to a (unique) solution $\psi$ fulfilling \eqref{regener} and the equation \eqref{lineq1} is satisfied in the following sense
 \be
A_0^{-1}(\beta \psi_{tt} + \psi_t) + \Delta^2 \psi + 2\Delta \psi + \Lambda A_0^{-1}\overline \psi= A_0^{-1} G
 \quad\hbox{ in } D(A_0^{-1}),\; \hbox{a.e. in}\, (0,T).\non
\ee
  On account of the above strong convergence, it is not difficult to show that $\psi$ satisfies the energy identity \eqref{enlin2}. Such an identity can be (formally) obtained by multiplying equation \eqref{lineq1} by $A^{-1}_0 \overline\psi_t$. The proof is complete.
\end{proof}

\textbf{Proof of Theorem \ref{exe}}. The proof relies on a suitable Faedo--Galerkin
approximation scheme.  We consider the eigenvalue problem $-\Delta w= \lambda w$ subject to periodic boundary conditions. It is well known that there exist two sequences $\{\lambda_n\}_{n=1,2,...}$ and $\{w_n\}_{n=1,2,...}$ such that, for every
$n\geq  1$, $\lambda_n\geq 0$ is an eigenvalue and $w_n\neq 0$ is a corresponding eigenfunction, the sequence ${\lambda_n}$ is
nondecreasing, tending to infinity as $n\to +\infty$, and the sequence $\{w_n\}$ is orthonormal and complete in $L^2_p(Q)$. We notice that
$\lambda = 0$ is an eigenvalue, whence $\lambda_1 = 0$, and that any non-zero constant is an eigenfunction (i.e., $w_1=1$). For every $i > 1$, $w_i$ cannot be a constant and $\langle w_i\rangle =0$, whence $\lambda_i=\int_Q|\nabla w_i|^2dx>0$. Moreover, as $w_1=1$ is a constant
and $\{w_n\}$ is orthonormal in $L^2_p(Q)$, we easily deduce that
$ A_0^{-1} w_i=\lambda_i^{-1} w_i$ for every $i>1$.

For any $n\geq 1$, we introduce the finite-dimensional space $W_n={\rm span}\{w_1,...,w_n\}$ and $\Pi_n$ the orthogonal projection on $W_n$. It is obvious that $W_n\subset C^\infty(Q)$. Then we consider the approximate problem $(P_n)$: looking for $t_n>0$ and $\xi_i\in C^2([0,t_n])$ such that $\phi^n(t):=\sum_{i=1}^n\xi_i(t)w_i$ solves
\bea&&
A_0^{-1}(\beta \phi^n_{tt}+\phi^n_t)+ \Delta ^2 \phi^n +2\Delta \phi^n+\Pi_n(f(\phi^n)-\langle f(\phi^n)\rangle)
=0, \label{e1aea}\\
&&
\phi^n|_{t=0}=\Pi_n\phi_0(x),\quad \phi^n_t|_{t=0}=\Pi_n \phi_1(x),\label{e2aea}
\eea
where both relations are intended as equalities in $W_n$.
It is easy to verify that Problem $(P_n)$ admits a unique global solution, which satisfies the energy estimate (cf. \eqref{disa1}) uniformly with respect to $n$ and time $t$. Standard compactness arguments (e.g., the Aubin-Lions lemma) permit to
take the limit of \eqref{e1aea}--\eqref{e2aea} at least for a subsequence. Thus, as a
consequence, we obtain existence of a global energy solution $(\phi, \phi_t)$ to problem \eqref{e1}--\eqref{e2}.

The regularity of energy solution is not sufficient for us to prove the uniqueness directly using the energy method. To this end, we use a non-standard argument developed in \cite{Se} (cf. also \cite{GSZ}). Let $(\phi, \phi_t)$ be any energy solution to problem \eqref{e1}--\eqref{e2}   on $[0,T]$, which is the limit of proper subsequence of the approximate solutions $(\phi^n, \phi^n_t)$. Define $(\Phi^n, \Phi_t^n)=\Pi_n (\phi, \phi_t)$. Then consider the projection of equation \eqref{e1} which, written for the difference $(v^n, v_t^n)=(\Phi^n-\phi^n, \Phi_t^n-\phi^n_t)$, reads
\bea
&& A_0^{-1}(\beta v^n_{tt}+v^n_t)+ \Delta ^2 v^n +2\Delta v^n\non\\
&&\qquad = -\Pi_n(f(\phi)-\langle f(\phi)\rangle-f(\phi^n)+\langle f(\phi^n)\rangle), \label{e1ad}\\
&&
v^n|_{t=0}=0,\quad v^n_t|_{t=0}=0.\label{e2ad}
\eea
It is obvious that the conservative relation holds so that
\be
\langle v^n(t)\rangle=\langle v^n_t(t)\rangle=0, \quad \forall\, t\geq 0.\label{dm}
\ee
Since $ (v^n, v_t^n) $ are regular, we can test \eqref{e1ad} by $A_0^{-1}v^n_t$. This entails
\bea
&& \frac{d}{dt}\left(\frac{\beta}{2}\|v^n_t\|_{-2}^2+ \frac12\|\nabla v^n\|^2-\|v^n\|_{-1}^2 \right)+\|v^n_t\|_{-2}^2\non\\
&=& \int_Q \Pi_n(f(\phi^n)-f(\Phi^n)) A_0^{-1}v^n_t dx+
\int_Q \Pi_n(f(\Phi^n)-f(\phi)) A_0^{-1}v^n_t dx\non\\
&\leq& \|f(\phi^n)-f(\Phi^n)\|\|v^n_t\|_{-2}+ \|f(\Phi^n)-f(\phi)\|_{-1}\|v^n_t\|_{-1}.
\label{d1d}
\eea
We recall that the $L^\infty(0,T;\mathbb{X}_0)$-norm of $(\Phi^n, \Phi_t^n)$ and $(\phi^n,\phi^n_t)$ are bounded (uniformly with respect to $n$) as well as $(\phi,\phi_t)$. Therefore, we have
 \bea
 \|f(\phi^n)-f(\Phi^n)\|
 &\leq& C \left\|\int_0^1 f'(\tau \phi^n+(1-\tau)\Phi^n) (\phi^n-\Phi^n)d\tau\right\|
 \non\\
 &\leq& C\|f'\|_{L^3_p(Q)}\|v^n\|_{L^6_p(Q)}\non\\
 &\leq& C\|\nabla v^n\|\non
 \eea
 and
 \bea
 &&\|f(\Phi^n)-f(\phi)\|_{-1}\|v^n_t\|_{-1}\non\\
 &\leq& C \left\|\int_0^1 f'(\tau \Phi^n+(1-\tau)\phi) (\Phi^n-\phi)d\tau\right\|_{L^{\frac65}_p(Q)}(\|\Phi^n_t\|_{-1}+\|\phi_t\|_{-1})\non\\
 &\leq& C\|f'\|_{L^3_p(Q)}\|\Phi^n-\phi\|\non\\
 &\leq& C\lambda_n^{-\frac12}\|\Phi^n-\phi\|_{1}\non\\
 &\leq&
  C\lambda_n^{-\frac12},\non
 \eea
 where the constant $C$ in the above estimates is independent of $n$. Using the Young inequality and the Poincar\'e inequality (cf. \eqref{dm}), we infer from the above estimates that
 \be \frac{d}{dt}\left(\frac{\beta}{2}\|v^n_t\|_{-2}^2+ \frac12\|\nabla v^n\|^2-\|v^n\|_{-1}^2 \right)+\frac12\|v^n_t\|_{-2}^2\leq C\|\nabla v^n\|^2+C\lambda_n^{-\frac12}
  \ee
  and integrating with respect to time, we get
  \bea
   && \frac{\beta}{2}\|v^n_t(t)\|_{-2}^2+ \frac12\|\nabla v^n(t)\|^2+\frac12 \int_0^t\|v^n_t(\tau)\|_{-2}^2d\tau \non\\
   &\leq& C\int_0^t \|\nabla v^n(\tau)\|^2d\tau +C\lambda_n^{-\frac12}t+ \|v^n(t)\|_{-1}^2.\label{aa1}
  \eea
 Observe now that
  \bea
  \|v^n(t)\|_{-1}^2&\leq& C\|v^n(t)\|\|v^n(t)\|_{-2}\non\\
  &\leq& C\|\nabla v^n(t)\|t^\frac12\left(\int_0^t\|v^n_t(\tau)\|_{-2}^2d\tau\right)^\frac12\non\\
  &\leq& \frac14\|\nabla v^n(t)\|^2+Ct\int_0^t\|v^n_t(\tau)\|_{-2}^2d\tau,\quad \forall\ t\in [0,T].
  \eea
  Thus we infer
  \bea
  && \frac{\beta}{2}\|v^n_t(t)\|_{-2}^2+ \frac14\|\nabla v^n(t)\|^2\non\\
  &\leq& \left(Ct-\frac12\right)\int_0^t\|v^n_t(\tau)\|_{-2}^2d\tau+C\int_0^t \|\nabla v^n(\tau)\|^2d\tau
  +C\lambda_n^{-\frac12}t\non\\
  &\leq& C(1+t)\int_0^t\left(\frac{\beta}{2}\|v^n_t(\tau)\|_{-2}^2+ \frac14\|\nabla v^n(\tau)\|^2\right)d\tau+C\lambda_n^{-\frac12}t\non
  \eea
  and by the Gronwall inequality we deduce
  \be
  \frac{\beta}{2}\|v^n_t(t)\|_{-2}^2+ \frac14\|\nabla v^n(t)\|^2\leq C \lambda_n^{-\frac12}t e^{C(1+t)t}.\non
  \ee
  Therefore, we obtain that
for $n\to +\infty$
 \be
 (v^n, v^n_t)\to 0, \quad \text{strongly in}\ L^\infty(0,T; H^1_p(Q)\times H^{-2}_p(Q)).
 \ee
Since we already know that $(\Phi^n, \Phi_t^n)\to (\phi, \phi_t)$ by definition, it follows by comparison that $(\phi^n, \phi^n_t)\to (\phi, \phi_t)$. Hence, the whole sequence $(\phi^n, \phi^n_t)$ converges to $(\phi, \phi_t)$ as $n\to +\infty$. This indeed entails the
uniqueness of  energy solution $(\phi, \phi_t)$.

Observe now that the unique energy solution can be seen as a solution to the linear
equation \eqref{lineq1} with
\be
G := A_0(\langle f(\phi)\rangle - f(\phi)) + \Lambda \overline \phi,
\label{FF}
\ee
where $\Lambda>0$ is sufficiently large. From a well-known embedding theorem due to J. Simon \cite{JS}
we see that the energy solution $\phi$ has the continuity property
$\phi\in C([0,T]; H^{2-\sigma}_p(Q))$ for $0<\sigma\ll 1$, which implies that $G \in C([0,T]; \dot{H}^{-1}_p(Q))$ (here we have used the Sobolev embedding $H^{2-\sigma}_p(Q)\hookrightarrow L^\infty_p(Q)$ for $n\leq 3$ and $\sigma$ small). Thus, we can conclude from Lemma \ref{lin} that \eqref{regener} holds. On the other hand, using a similar approximation procedure as in the proof of Lemma \ref{lin}, we also infer that $\phi$ fulfills the following energy identity
\bea
{\cal E}_0(\phi(t),\phi_t(t)) &=& {\cal E}_0(\phi(s),\phi_t(s)) - \int_s^t \Vert \overline{ \phi_t}(\tau)\Vert_{-1}^2 d\tau\cr \non\\
&&+ \int_s^t \int_Q A_0^{-\frac12} G(\tau) A^{-\frac12 }_0 \overline{\phi_t}(\tau) dx d\tau, \quad \forall\, s, t\in [0,T],\, s<t,\label{enlin2a}
\eea
where $G$ is given by \eqref{FF}. On the other hand, we have
\bea
&& \int_s^t \int_Q A_0^{-\frac12} G(\tau) A^{-\frac12 }_0 \overline{\phi_t}(\tau) dx d\tau\non\\
&=& \int_s^t \int_Q(\langle f(\phi(\tau))\rangle - f(\phi(\tau)))  \overline{\phi_t}(\tau) dx d\tau + \frac{\Lambda}{2} \|\overline \phi(t)\|^2-\frac{\Lambda}{2} \|\overline {\phi}(s)\|^2\non\\
&=& -\int_Q F(\phi(t)) dx+ \int_Q F(\phi(s)) dx+ \int_s^t \mathcal{A}(\tau)\int_Q f(\phi(\tau)) dxd\tau\non\\
&&\quad + \frac{\Lambda}{2} \|\overline \phi(t)\|^2-\frac{\Lambda}{2} \|\overline {\phi}(s)\|^2,\non
\eea
so that \eqref{enereq} follows from \eqref{enlin2a}. The proof of Theorem \ref{exe} is complete. $\square$

\textit{The dynamical system}. We can now associate with problem \eqref{e1}--\eqref{e2}   a semiflow $\mathcal{S}=(\mathbb{X}_0,S(t))$  where $S(t): \mathbb{X}_0\to \mathbb{X}_0$ is the semigroup defined as follows
 $$S(t)(\phi_0, \phi_1)=(\phi(t), \phi_t(t)), \quad \forall \, t\geq 0,$$
$\phi$ being the energy solution given by Theorem \ref{exe}.

From the energy identity satisfied by the energy solution, we are able to check that $S(t)$ is indeed uniformly Lipschitz continuous on bounded balls of $\mathbb{X}_0$, for any fixed $t\geq 0$. Let us consider two pairs of initial data $(\phi_{0j}, \phi_{1j})$ ($j=1,2$) in a bounded set of $\mathbb{X}_0$ and write the corresponding equation for $\tilde{\phi}=\phi_1-\phi_2$, $\phi_j$ being the solution corresponding to the initial datum $(\phi_{0j}, \phi_{1j})$.
Then, we can take advantage of the associated energy identity with $G:=A_0(\langle f(\phi_1)\rangle -\langle f(\phi_2)\rangle - f(\phi_1) +f(\phi_2)) + \Lambda (\overline \phi_1 - \overline \phi_2)$ as a source term in \eqref{enlin2} so that
\bea
{\cal E}_0(\tilde{\phi}(t),\tilde{\phi}_t(t)) &=& {\cal E}_0(\tilde{\phi}_0,\tilde{\phi}_1) - \int_0^t \Vert \overline{ \tilde{\phi}_t}(\tau)\Vert_{-1}^2 d\tau  + \int_0^t \int_Q A_0^{-\frac12} G(\tau) A^{-\frac12 }_0 \overline{\tilde{\phi}_t}(\tau) dx d\tau\non\\
&\leq& {\cal E}_0(\tilde{\phi}_0,\tilde{\phi}_1)+ C \int_0^t \|A_0^{-\frac12} G(\tau)\|^2 d\tau \non\\
&\leq& {\cal E}_0(\tilde{\phi}_0,\tilde{\phi}_1)+ C \int_0^t \|\tilde{\phi}(\tau)\|_1^2 d\tau\non\\
&\leq& {\cal E}_0(\tilde{\phi}_0,\tilde{\phi}_1)+ C \int_0^t {\cal E}_0(\tilde{\phi}(\tau),\tilde{\phi}_t(\tau)) d\tau+ C\int_0^t |\langle\tilde{\phi}(\tau)\rangle|^2 d\tau.\non
\eea
On the other hand, we have
$\langle \tilde{\phi}(t)\rangle = \beta\langle \tilde{\phi}_1\rangle+\langle \tilde{\phi}_0\rangle-\beta \langle \tilde{\phi}_1\rangle e^{-\frac{t}{\beta}}$.
As a consequence, we obtain
\bea
&&{\cal E}_0(\tilde{\phi}(t),\tilde{\phi}_t(t))+ |\langle\tilde{\phi}(t)\rangle|^2+\int_0^t \Vert \overline{ \tilde{\phi}_t}(\tau)\Vert_{-1}^2 d\tau \non\\
&\leq& {\cal E}_0(\tilde{\phi}_0,\tilde{\phi}_1)+C|\langle \tilde{\phi}(0)\rangle|^2+ C \int_0^t {\cal E}_0(\tilde{\phi}(\tau),\tilde{\phi}_t(\tau)) + |\langle\tilde{\phi}(\tau)\rangle|^2 d\tau.\non
\eea
Thus, by the Gronwall inequality, \eqref{EE0} and the fact $\langle \tilde{\phi}(t)\rangle = \beta\langle \tilde{\phi}_1\rangle+\langle \tilde{\phi}_0\rangle-\beta \langle \tilde{\phi}_1\rangle e^{-\frac{t}{\beta}}$, we conclude that
\bea
&&
\|(\phi_1-\phi_2, \phi_{1t}-\phi_{2t})(t)\|_{\mathbb{X}_0}^2+ \int_0^t \Vert ( \phi_{1t}-\phi_{2t}) (\tau)\Vert_{-1}^2 d\tau\non\\
&\leq& L_1 e^{L_2t} \|(\phi_{10}-\phi_{20}, \phi_{11}-\phi_{21})\|_{\mathbb{X}_0}^2 ,\label{LipX0}
\eea
where $L_1$, $L_2$ are positive constants depending on the $\mathbb{X}_0$-norms of the initial data as well as on $\beta$, $\epsilon$, $Q$ and $f$.

Next, due to the dissipative estimate in Lemma \ref{es}, we can state some dissipative properties of the dynamical system defined on a suitable phase space. Recalling the conservative property \eqref{conODE2}, we have to work on the following subset of $\mathbb{X}_0$:   $$\mathbb{X}^{M,M'}_0=\{(u,v)\in \mathbb{X}_0: \ |\beta\langle v\rangle+\langle u\rangle|\leq M, \ |\langle v\rangle|\leq M'\}.$$

\begin{theorem}\label{abs}
The semiflow $\mathcal{S}$ is
uniformly dissipative on the phase space $\mathbb{X}_0^{M,M'}$. Namely, there exists a constant $R_0$ independent of the initial data such that, for all bounded set $B_0\subset \mathbb{X}_0^{M,M'}$,  there exists a $t_{B_0}>0$ such that, for all $(\phi_0, \phi_1)\in \mathbb{X}_0^{M,M'}$,
 \be
\|S(t)(\phi_0, \phi_1)\|_{\mathbb{X}_0}\leq R_0,\quad \forall\,t \geq t_{B_0}.\non
\ee
\end{theorem}

For more regular initial data, for instance, $(\phi_0, \phi_1)\in \mathbb{X}_1$,
 we can deduce from Lemma \ref{es1} and a similar Galerkin approximation scheme the existence and uniqueness of weak solutions to problem \eqref{e1}--\eqref{e2}   as well as the existence of an absorbing set. More precisely, the following results hold:
 \begin{theorem}\label{exe1}
  For any initial data $(\phi_0, \phi_1)\in \mathbb{X}_1$, problem \eqref{e1}--\eqref{e2}   admits a unique global weak solution $(\phi, \phi_t)$.
 \end{theorem}
 \begin{theorem}\label{abs1}
The semiflow $\mathcal{S}$ is
uniformly dissipative on the phase space $$\mathbb{X}^{M,M'}_1=\{(u,v)\in \mathbb{X}_1: \ |\beta\langle v\rangle+\langle u\rangle|\leq M, \ |\langle v\rangle|\leq M'\}.$$ Namely, there exists a constant $R_1$ independent of the initial data such that, for all bounded set $B_1\subset \mathbb{X}_1^{M,M'}$,  there exists a $t_{B_1}>0$ such that for all $(\phi_0, \phi_1)\in \mathbb{X}_1^{M,M'}$,
 \be
\|S(t)(\phi_0, \phi_1)\|_{\mathbb{X}_1}\leq R_1,\quad \forall\,t \geq t_{B_1}.\non
\ee
\end{theorem}

\section{Global and exponential attractors}
\setcounter{equation}{0}
\noindent
In this section, we study the (global) longtime behavior of the semiflow $\mathcal{S}$ associate with the MPFC equation \eqref{e1}--\eqref{e2}.

\subsection{Global attractor}
The existence of the global attractor of problem \eqref{e1}--\eqref{e2} in the phase space $\mathbb{X}_0^{M,M'}$ is given by

\bt\label{gloatt}
For each $\beta>0$ and $M, M'>0$, the semiflow $\mathcal{S}$ defined on the phase space $\mathbb{X}_0^{M,M'}$ has a connected global attractor $\mathcal{A}_0$, which is bounded in $\mathbb{X}_1$.
\et

The basic step in the proof of Theorem \ref{gloatt} is to show certain (pre)compactness property of trajectories in the phase space $\mathbb{X}_0$. We have to overcome the difficulties from the hyperbolic nature of the system.
To this end, we establish a proper decomposition of the semigroup $S(t)$ into an uniformly asymptotically stable part and a compact part. We note that this decomposition also entails that the attractor $\mathcal{A}_0$ is bounded in the more regular space $\mathbb{X}_1$. This fact will be further exploited for constructing an exponential attractor.

Let $(\phi, \phi_t)$ be the unique energy solution to problem \eqref{e1}--\eqref{e2}   given in Theorem \ref{exe}. We split this solution
into two parts, namely,
 $$  (\phi, \phi_t)(t)= (\phi^d, \phi^d_t)(t)+ (\phi^c, \phi^c_t)(t),$$
where
 \bea
&& A_0^{-1}(\beta \phi^d_{tt}+\phi^d_t)+ \Delta^2\phi^d+ 2\Delta \phi^d+f_k(\phi^d)-\langle f_k(\phi^d)\rangle=0,\label{e1d}
\\
&&\phi^d|_{t=0}=\overline{{\phi}_0}(x),\quad \phi^d_t|_{t=0}=\overline{\phi_1}(x), \label{e2d}
\eea
 and
 \bea
&& A_0^{-1}(\beta \phi^c_{tt}+\phi^c_t)+\Delta^2 \phi^c+2\Delta \phi^c+f_k(\phi)-f_k(\phi-\phi^c)\non\\
&&\quad -\langle f_k(\phi)\rangle+\langle f_k(\phi-\phi^c)\rangle
\non\\
&=&k\phi-k\langle\phi\rangle,\label{e1c}
\\
&&\phi^c|_{t=0}=\langle\phi_0(x)\rangle,\quad \phi^c_t|_{t=0}=\langle\phi_1(x)\rangle.\label{e2c}
\eea
Here, we set
$$f_k(\phi):=f(\phi)+k\phi$$ with $k>0$ being a sufficiently large constant to be chosen later. In particular,
we require that $f_k(s)$ is monotone and nondecreasing in $\mathbb{R}$.

\bl\label{decay}
Let the assumptions of Theorem \ref{exe} hold. Then there exists a sufficiently large $k$ such that
\be
\|(\phi^d(t), \phi_t^d(t))\|_{\mathbb{X}_0}\leq C(\|(\overline{\phi_0}(x), \overline{\phi_1}(x))\|_{\mathbb{X}_0})e^{-\kappa t}, \quad \forall\, t\geq 0,\label{decayx0}
\ee
where $\kappa>0$ is a small constant.
\el
\begin{proof}
For any positive constant $k$, the existence and uniqueness of a global energy solution $(\phi^d(t), \phi_t^d(t))$ to problem \eqref{e1d}--\eqref{e2d} easily follows through the same argument used to prove Theorem \ref{exe}. Moreover, due to the zero-mean assumption on the initial data \eqref{e2d}, we conclude that $$\langle\phi^d(t)\rangle=\langle\phi_t^d(t)\rangle=0, \quad \forall\, t\geq 0,$$ which also yields $\phi^d(t)=\overline{\phi^d}(t)$ and $ \phi^d_t(t)=\overline{\phi^d_t}(t)$.

Testing \eqref{e1d} by $\phi^d_t$ and $\phi^d$, respectively, we have
\be
\frac{d}{dt}\left(\frac\beta2\|\phi^d_t\|_{-1}^2+ \frac12\|\Delta \phi^d\|^2-\|\nabla \phi^d\|^2+ \int_{Q} F_k(\phi^d) dx\right)+\|\phi^d_t\|_{-1}^2=0,\label{decayx01}
\ee
and
\bea
&& \frac{d}{dt}\left(\beta (\phi^d_t, \phi^d)_{-1}+\frac12\|\phi^d\|_{-1}^2\right)-\beta \|\phi^d_t\|^2_{-1}\non\\
&&\qquad +\|\Delta \phi^d\|^2-2\|\nabla \phi^d\|^2+\int_{Q} f_k(\phi^d)  \phi^d dx = 0,\label{decayx02}
\eea
where $$F_k(\phi^d)=\frac{1-\epsilon+k}{2}(\phi^d)^2+\frac14 (\phi^d)^4.$$
Multiplying \eqref{decayx02} by $\eta>0$ and adding it to \eqref{decayx01}, we get
\be
\frac{d}{dt}\mathcal{Y}^d_1(t)+\mathcal{D}^d_1(t)\leq 0,\label{decayx03}
\ee
where
\bea
\mathcal{Y}^d_1(t)&=& \frac{\beta}{2}\|\phi^d_t\|_{-1}^2+ \frac12\|\Delta \phi^d\|^2-\|\nabla \phi^d\|^2+ \int_{Q} F_k(\phi^d) dx\non\\
&&\quad +  \eta \beta (\phi^d_t, \phi^d)_{-1}+\frac{\eta}{2}\| \phi^d\|_{-1}^2,
\nonumber\\
\mathcal{D}^d_1(t)&=& (1-\eta\beta)\|\phi^d_t\|_{-1}^2+\eta \|\Delta \phi^d\|^2-2\eta \|\nabla \phi^d\|^2+\eta \int_{Q} f_k(\phi^d) \phi^d dx.\non
\eea
Recalling \eqref{intpo}, we take $k>0$ sufficiently large and $\eta>0$ sufficiently small such that
\be
C(\|\phi^d_t\|_{-1}, \|\phi^d\|_{H^2})\geq \mathcal{Y}^d_1(t)\geq \frac{\beta}{4}\|\phi^d_t\|_{-1}^2+\frac14\|\Delta \phi^d\|^2+ \frac{k}{2}\|\phi^d\|^2,\label{Yd1}
\ee
\be
 \mathcal{D}^d_1(t)\geq \frac12 \|\phi^d_t\|_{-1}^2+\frac\eta2 \|\Delta \phi^d\|^2+\frac\eta2 \int_{Q} f_k(\phi^d) \phi^d dx,\non
\ee
and
\be
\kappa\mathcal{Y}^d_1(t)\leq \mathcal{D}^d_1(t),\non
\ee
where
$\kappa$ is a (small) constant independent of  $\phi^d$ and time $t$.
As a result, we have
\be
\frac{d}{dt}\mathcal{Y}^d_1(t)+\kappa\mathcal{Y}^d_1(t)\leq 0,\label{decayx04}
\ee
which implies
\be
\mathcal{Y}^d_1(t)\leq \mathcal{Y}^d_1(0)e^{-\kappa t}.
\ee
We infer from \eqref{Yd1} that \eqref{decayx0} holds. The proof is complete.
\end{proof}

\bl\label{com}
Let the assumptions of Theorem \ref{exe} hold. Then we have
\be
\|(\phi^c(t), \phi_t^c(t))\|_{\mathbb{X}_1}\leq C(\|(\phi_0, \phi_1)\|_{\mathbb{X}_0}), \quad \forall\, t\geq 0.\label{comx0}
\ee
\el
\begin{proof}
Let the constant $k$ be the one we choose in Lemma \ref{decay}. For the initial data $(\phi_0, \phi_1)$ belonging to a bounded set in $\mathbb{X}_0$, it follows from the uniform estimates \eqref{disa1} and \eqref{decayx0} that $(\phi^c, \phi^c_t)$ also belongs to a bounded set in $\mathbb{X}_0$, i.e.,
 \be
 \|(\phi^c(t), \phi_t^c(t))\|_{\mathbb{X}_0}\leq C(\|(\phi_0, \phi_1)\|_{\mathbb{X}_0}), \quad \forall\, t\geq 0.\label{comx1}
 \ee
 To prove the fact that $(\phi^c(t), \phi_t^c(t))$ is indeed more regular, we perform some higher-order calculations that can be justified rigorously by working within a proper Galerkin scheme as before.

Testing \eqref{e1c} by $A_0 \overline{\phi^c_t}$ and $A_0 \overline{\phi^c}$, respectively, we get
 \bea
 && \frac{d}{dt}\left(\frac{\beta}{2}\|\overline{\phi^c_t}\|^2+\frac12
 \|\nabla \Delta \phi^c\|^2-\|\Delta \phi^c\|^2+\frac{k}{2}\|\nabla \phi^c\|^2\right)+\|\overline{\phi^c_t}\|^2\non\\
 &=& \int_Q\Delta(f(\phi)-f(\phi-\phi^c))\overline{\phi^c_t} dx-k\int_Q \Delta \phi \overline{\phi^c_t} dx,\label{comx2}
 \eea
 \bea
 &&\frac{d}{dt}\left(\beta(\overline{\phi^c_t}, \overline{\phi^c})+\frac12\|\overline{\phi^c}\|^2\right)-\beta\|\overline{\phi^c_t}\|^2+\|\nabla \Delta \phi^c\|^2-2\|\Delta \phi^c\|^2+k\|\nabla \phi^c\|^2\non\\
 &=&\int_Q(f(\phi)-f(\phi-\phi^c))\Delta \phi^c dx-k\int_Q \phi \Delta \phi^c dx.\label{comx3}
 \eea
 Multiplying \eqref{comx3} by $\eta_1>0$ and adding it to \eqref{comx2}, we deduce
 \be
\frac{d}{dt}\mathcal{Y}^c_1(t)+\mathcal{D}^c_1(t)\leq \mathcal{R}^c_1(t),\label{comx4}
\ee
where
\bea
\mathcal{Y}^c_1(t)&=& \frac{\beta}{2}\|\overline{\phi^c_t}\|^2+\frac12
 \|\nabla \Delta \phi^c\|^2-\|\Delta \phi^c\|^2+\frac{k}{2}\|\nabla \phi^c\|^2+\eta_1\beta(\overline{\phi^c_t}, \overline{\phi^c})+\frac{\eta_1}{2}\|\overline{\phi^c}\|^2,\non\\
 \mathcal{D}^c_1(t)&=&(1-\eta_1\beta)\|\overline{\phi^c_t}\|^2+\eta_1\|\nabla \Delta \phi^c\|^2-2\eta_1\|\Delta \phi^c\|^2+\eta_1k\|\nabla \phi^c\|^2,\non\\
 \mathcal{R}^c_1(t)&=&\int_Q\Delta(f(\phi)-f(\phi-\phi^c))\overline{\phi^c_t} dx+\eta_1\int_Q(f(\phi)-f(\phi-\phi^c))\Delta \phi^c dx\non\\
 && \quad -k\int_Q \Delta \phi \overline{\phi^c_t} dx-\eta_1k\int_Q \phi \Delta \phi^c dx.\non
\eea
First, for sufficiently large $k$ and small $\eta_1$, we can easily see that
 \be
 \mathcal{D}^c_1(t)\geq \frac12\|\overline{\phi_t^c}\|^2+\kappa'\mathcal{Y}^c_1(t)\geq C(\|\overline{\phi^c_t}\|^2+
 \|\nabla \phi^c\|_{H^2}^2),\label{comx7}
 \ee
 where $\kappa'>0$ is a small constant.

  Due to the Sobolev embedding $H^2(Q) \hookrightarrow L^\infty(Q)$ ($n\leq 3$), the remainder term $\mathcal{R}^c_1$ can be estimated by using the uniform estimates of $(\phi, \phi_t)$ and $(\phi^c, \phi^c_t)$ in the $\mathbb{X}_0$-norm (see \eqref{disa1} and \eqref{comx1}) such that
 \bea
 \mathcal{R}_1^c(t)&\leq&
  \|\overline{\phi^c_t} \| \|\Delta(f(\phi)-f(\phi-\phi^c))\|+\eta_1\|f(\phi)-f(\phi-\phi^c)\|\|\Delta \phi^c\|\non\\
 && + k\|\Delta \phi\|\|\overline{\phi^c_t} \|+\eta_1 k\|\phi\|\|\Delta \phi^c\|\non\\
 &\leq& \frac12\|\overline{\phi_t^c}\|^2+ \|\Delta(f(\phi)-f(\phi-\phi^c))\|^2+ k^2\|\Delta \phi\|^2\non\\
 &&+\eta_1\|f(\phi)-f(\phi-\phi^c)\|\|\Delta \phi^c\| +\eta_1 k\|\phi\|\|\Delta \phi^c\|\non\\
 &\leq& \frac12\|\overline{\phi_t^c}\|^2+ C(\|\phi\|_{2}, \|\phi^c\|_{2})\non\\
 &\leq&  \frac12\|\overline{\phi_t^c}\|^2+ C(\|(\phi_0, \phi_1)\|_{\mathbb{X}_0}).\non
 \eea
 The above estimate combined \eqref{comx4} and \eqref{comx7} yields that
 \be
\frac{d}{dt}\mathcal{Y}^c_1(t)+\kappa'\mathcal{Y}^c_1(t)\leq C(\|(\phi_0, \phi_1)\|_{\mathbb{X}_0}).\label{comx5}
\ee
 As a result, we find
 \be
 \mathcal{Y}^c_1(t)\leq \mathcal{Y}^c_1(0)e^{-\kappa' t}+ \frac{C(\|(\phi_0, \phi_1)\|_{\mathbb{X}_0}) }{\kappa'},\quad \forall t\geq 0.\label{comx6}
 \ee
 On the other hand, the choice of initial data \eqref{e2c} indicates that $\mathcal{Y}_1^c(0)=0$.
 Then, from \eqref{comx7} and \eqref{comx6}  we conclude that \eqref{comx0} holds. The proof is complete.
\end{proof}

\textbf{Proof of Theorem \ref{gloatt}}. We have shown that a given trajectory originating from $\mathbb{X}_0$ is a sum of an exponentially
decaying part and a term that belongs to a closed bounded subset of the more regular space $\mathbb{X}_1$. Therefore,
the trajectory is precompact in $\mathbb{X}_0$. On the other hand, on account of Theorem \ref{abs}, the semigroup $S(t)$ has a bounded attracting
set in $\mathbb{X}_0^{M,M'}$ (for any $\beta>0$). Hence, the conclusion of Theorem \ref{gloatt} follows from a well-known abstract result for infinite dimensional dynamical systems (see, e.g., \cite[Theorem 1.1]{Te}). $\square$

\subsection{Exponential attractors}
In what follows, we proceed to prove the existence of an exponential attractor for the
semiflow $\mathcal{S}$ consisting of energy solutions to problem \eqref{e1}--\eqref{e2}.
For the importance of this notion the reader is referred to \cite{MZ} and references therein.
More precisely, we will prove

\bt \label{exat} For each $\beta>0$ and $M, M'>0$, the semiflow $\mathcal{S}$ defined on the phase space $\mathbb{X}_0^{M,M'}$ admits an
exponential attractor $\mathcal{M}_0$, which is a positively invariant, compact subset of $\mathbb{X}_0$
with finite fractal dimension with respect to the $\mathbb{X}_0$-metric and bounded in $\mathbb{X}_1$, such that,
for any bounded $\mathcal{B}\subset \mathbb{X}_0^{M,M'}$ there exist $K_\mathcal{B} > 0$ and $\gamma_\mathcal{B}> 0$ such that
\be
{\rm dist}_{\mathbb{X}_0}(S(t)\mathcal{B}, \mathcal{M}_0)\leq K_{\mathcal{B}}e^{-\gamma_\mathcal{B} t},\label{exp}
\ee
where
${\rm dist}_{\mathbb{X}_0}$ denotes the Hausdorff semidistance of sets with respect to the $\mathbb{X}_0$-metric.
\et
We note that by its definition an exponential attractor $\mathcal{M}_0$ contains the global attractor $\mathcal{A}_0$ obtained in Theorem \ref{gloatt}.
As a consequence, we have
\begin{corollary}\label{GA1}
 The global attractor $\mathcal{A}_0$ has finite fractal dimension.
\end{corollary}

Different approaches can be employed to prove the existence of an exponential
attractor $\mathcal{M}_0$ to problem \eqref{e1}--\eqref{e2}. Here, we shall use the simple constructive method introduced in
\cite[Proposition 1]{EMZ}. The procedure consists of three basic steps.

\textit{Step 1. Confining the dynamics on a regular positively invariant set in $\mathbb{X}_1$}.
Our previous results yield the following preliminary observations:

\bp \label{abs2} Let the assumptions of Theorem \ref{exat} are satisfied. Then:

(i) there exists a bounded set $\mathcal{B}_1$ in $\mathbb{X}_1$ that exponentially attracts any bounded set of $\mathbb{X}_0^{M,M'}$ with respect to the $\mathbb{X}_0$-metric;

(ii) there exists a bounded positively invariant set $\mathcal{V}_1$ in
$\mathbb{X}_1$, which absorbs the set $\mathcal{B}_1$ and, consequently, exponentially attracts
any bounded set of $\mathbb{X}_0^{M,M'}$ with respect to the $\mathbb{X}_0$-metric.
\ep
\begin{proof} The conclusion (i) is a simple consequence of Lemmas \ref{decay} and \ref{com}. As far as (ii) is concerned, we first recall Lemma \ref{es1}, which gives a dissipative estimate on $\mathcal{B}_1$ (cf. also Theorem \ref{abs1}). This entails the existence of a positively invariant and
$\mathbb{X}_1$-bounded set $\mathcal{V}_1$, which eventually absorbs any $\mathbb{X}_1$-bounded set of data. In particular, $\mathcal{V}_1$ absorbs $\mathcal{B}_1$, and by the definition of $\mathcal{B}_1$ in (i), we arrive at (ii). The proof is complete.
\end{proof}
\textit{Step 2. Existence of a smoother exponentially attracting set in $\mathbb{X}_1$}.
Let us take initial data lying in the (regular and positively invariant) set $\mathcal{V}_1$
constructed in Proposition \ref{abs2}. Notice also that it is not restrictive to assume $\mathcal{V}_1$ to be weakly closed in $\mathbb{X}_1$.
We now show the asymptotic smoothing property and H\"{o}lder continuity of the semigroup $S(t)$ on $\mathcal{V}_1$.

\begin{lemma}
\label{lmassm} Denote $z=(\phi, \phi_t)$. There exists $t^*\geq 0$ such that, setting $\mathrm{S}=S(t^*)$, we have
$$ \mathrm{S}z_{01}-\mathrm{S}z_{02}=D(z_{01},z_{02})+K(z_{01},z_{02}),$$
for every $z_{01}$, $z_{02}\in \mathcal{V}_1$, where $D$ and $K$ satisfy
\be
\|D(z_{01},z_{02})\|_{\mathbb{X}_0}\leq\lambda\|z_{01}-z_{02}\|_{\mathbb{X}_0},\quad
\|K(z_{01},z_{02})\|_{\mathbb{X}_1}\leq \Lambda\|z_{01}-z_{02}\|_{\mathbb{X}_0},\label{asymsmoo}
\ee
for some $\lambda\in(0,\frac{1}{2})$ and $\Lambda\geq 0$.
\end{lemma}
\begin{proof}
For any $z_{01}$, $z_{02}\in \mathcal{V}_1$, we simply denote $(\phi_i, \phi_{it})(t)=S(t)z_{0i}$ ($i=1,2$) the weak solutions to the MPFC equation \eqref{e1} with corresponding initial data and
$$z(t)=S(t)z_{01}-S(t)z_{02}=(\psi, \psi_t)(t),\quad  z_0=z_{01}-z_{02}=(\psi_0, \psi_1).$$
As before, we write the difference of solution $(\phi, \phi_t)$ as follows
$$  (\psi, \psi_t)(t)= (\psi^d, \psi^d_t)(t)+ (\psi^c, \psi^c_t)(t),$$
where
 \bea
&& A_0^{-1}(\beta \psi^d_{tt}+\psi^d_t)+ \Delta^2\psi^d+ 2\Delta \psi^d+k \psi^d=0,\label{e1dd}
\\
&&\phi^d|_{t=0}=\overline{{\psi}_0}(x),\quad \phi^d_t|_{t=0}=\overline{\psi_1}(x), \label{e2dd}
\eea
 and
 \bea
&& A_0^{-1}(\beta \psi^c_{tt}+\psi^c_t)+\Delta^2 \psi^c+2\Delta \psi^c +f(\phi_1)-\langle f(\phi_1)\rangle\non\\
&&\quad -f(\phi_2)+\langle f(\phi_2)\rangle \non\\
&=&k(\psi-\psi^c),\label{e1cd}
\\
&&\psi^c|_{t=0}=\langle\psi_0(x)\rangle,\quad \psi^c_t|_{t=0}=\langle\psi_1(x)\rangle.\label{e2cd}
\eea
Here, $k>0$ is again a sufficiently large constant (not necessarily the same used in the previous decomposition). For large $k$, it is easy to show the decay of $\psi^d$, which can be viewed as the solution to the linear problem \eqref{e1dd}--\eqref{e2dd}:
 \be
\|\psi^d(t)\|_2^2 +\beta\|\psi_t^d(t))\|_{-1}^2\leq C\|(\overline{\psi_0}(x), \overline{\psi_1}(x))\|_{\mathbb{X}_0}^2e^{-\kappa t}, \quad \forall\, t\geq 0.\label{decayx0d}
 \ee
 Next, testing \eqref{e1cd} by $A_0 \overline{\psi^c_t}$ and $A_0 \overline{\psi^c}$, respectively, we get
 \bea
 && \frac{d}{dt}\left(\frac{\beta}{2}\|\overline{\psi^c_t}\|^2+\frac12
 \|\nabla \Delta \psi^c\|^2-\|\Delta \psi^c\|^2+\frac{k}{2}\|\nabla \psi^c\|^2\right)+\|\overline{\psi^c_t}\|^2\non\\
 &=& \int_Q\Delta(f(\phi_1)-\langle f(\phi_1)\rangle -f(\phi_2)+\langle f(\phi_2)\rangle)\overline{\psi^c_t} dx-k\int_Q \Delta \psi \overline{\psi^c_t} dx,\label{comx2d}
 \eea
 \bea
 &&\frac{d}{dt}\left(\beta(\overline{\psi^c_t}, \overline{\psi^c})+\frac12\|\overline{\psi^c}\|^2\right)-\beta\|\overline{\psi^c_t}\|^2+\|\nabla \Delta \psi^c\|^2-2\|\Delta \psi^c\|^2+k\|\nabla \psi^c\|^2\non\\
 &=&\int_Q(f(\phi_1)-\langle f(\phi_1)\rangle-f(\phi_2)+\langle f(\phi_2)\rangle)\Delta \psi^c dx-k\int_Q \psi \Delta \psi^c dx.\label{comx3d}
 \eea
 Similar to \eqref{comx4}, we have
 \be
\frac{d}{dt}\mathcal{Y}^c_2(t)+\mathcal{D}^c_2(t)\leq \mathcal{R}^c_2(t),\label{comx4d}
\ee
where
 $\eta_2>0$ and
  \bea
\mathcal{Y}^c_2(t)&=& \frac{\beta}{2}\|\overline{\psi^c_t}\|^2+\frac12
 \|\nabla \Delta \psi^c\|^2-\|\Delta \psi^c\|^2+\frac{k}{2}\|\nabla \psi^c\|^2+\eta_2\beta(\overline{\psi^c_t}, \overline{\psi^c})+\frac{\eta_2}{2}\|\overline{\psi^c}\|^2,\non\\
 \mathcal{D}^c_2(t)&=&(1-\eta_2\beta)\|\overline{\psi^c_t}\|^2+\eta_2\|\nabla \Delta \psi^c\|^2-2\eta_2\|\Delta \psi^c\|^2+\eta_2 k\|\nabla \psi^c\|^2,\non\\
 \mathcal{R}^c_2(t)&=&\int_Q\Delta(f(\phi_1)-\langle f(\phi_1)\rangle -f(\phi_2)+\langle f(\phi_2)\rangle)\overline{\psi^c_t} dx-k\int_Q \Delta \psi \overline{\psi^c_t} dx\non\\
 && \quad +\eta_2 \int_Q(f(\phi_1)-\langle f(\phi_1)\rangle-f(\phi_2)+\langle f(\phi_2)\rangle)\Delta \psi^c dx-\eta_2k\int_Q \psi \Delta \psi^c dx\non
\eea
The argument used to get \eqref{comx7} easily yields that, for sufficiently large $k$ and small $\eta_2$,
 \be
 \mathcal{D}^c_2(t)\geq \frac12\|\overline{\psi_t^c}\|^2+\frac{\eta_2k}{2} \|\nabla \psi^c\|^2+\kappa''\mathcal{Y}^c_2(t)\geq C(\|\overline{\psi^c_t}\|^2+
 \|\nabla \psi^c\|_{H^2}^2).\label{comx7d}
 \ee
 Finally,  using the uniform $\mathbb{X}_0$-estimates of $(\phi_i, \phi_{it})$ and $(\psi^c, \psi^c_t)$,   the remainder $\mathcal{R}^c_2$ can be estimated by
 \bea
 \mathcal{R}_2^c(t)
 &\leq& \frac12\|\overline{\psi_t^c}\|^2+\frac{\eta_2k}{2} \|\nabla \psi^c\|^2+ C(\|\phi_1\|_{2}, \|\phi_2\|_{2})\|\psi\|_2^2,\non
 \eea
 which implies
 \be
\frac{d}{dt}\mathcal{Y}^c_2(t)+\kappa''\mathcal{Y}^c_2(t)\leq C(\|\phi_1\|_{2}, \|\phi_2\|_{2})\|\psi\|_2^2.\label{comx5d}
\ee
Integrating \eqref{comx5d} with respect to time, we infer from the choice of initial data and the Lipschitz continuity estimate \eqref{LipX0} that
 \bea
 \mathcal{Y}^c_2(t)&\leq& \mathcal{Y}^c_2(0)+ \int_0^t C(\|\phi_1(s)\|_{2}, \|\phi_2(s)\|_{2})\|\psi(s)\|_2^2ds \non\\
 &\leq&  C(t)\|(\psi_0, \psi_1)\|_{\mathbb{X}_0}^2.\label{KK}
 \eea
 Due to \eqref{decayx0d}, for any fixed $\lambda\in (0,\frac12)$,
 we can choose $t^*$ sufficiently large such
 that
 \be
 \|(\psi^d(t^*), \psi_t^d(t^*))\|_{\mathbb{X}_0}\leq \lambda \|(\psi_0(x), \psi_1(x))\|_{\mathbb{X}_0}.\label{DD}
 \ee
 Fix such $t^*$ and set
 \be
 {\rm S}=S(t^*), \quad D(z_{01},z_{02})=(\psi^d(t^*), \psi_t^d(t^*)),\quad
 K(z_{01},z_{02})=(\psi^c(t^*),\psi^c_t(t^*)).\non
 \ee
 It follows from \eqref{KK} and \eqref{DD} that \eqref{asymsmoo} holds. The proof is complete.
\end{proof}

\begin{lemma}
\label{lmholder}
Denote $z=(\phi, \phi_t)$. For any $t^*> 0$, the map $(t,z)\mapsto S(t)z: [t^*,2t^*]\times\mathcal{V}_1\rightarrow \mathcal{V}_1 $
 is $\frac{1}{3}$-H\"{o}lder continuous in time and Lipschitz continuous in the initial
 data, when $\mathcal{V}_1$ is endowed with the $\mathbb{X}_0$-topology.
\end{lemma}
\begin{proof}
For any $t,\tau\in [t^*, 2t^*]$ satisfying $t\geq \tau$ and $z_1, z_2\in \mathcal{V}_1$,
we have
\be
\|S(t)z_1-S(\tau)z_2\|_{\mathbb{X}_0}^2\leq 2 \|S(t) z_1-S(t)z_2\|_{\mathbb{X}_0}^2+ 2 \|S(t) z_2-S(\tau)z_2\|_{\mathbb{X}_0}^2,\label{holder}
\ee
where the first term on the right-hand side can be easily estimated like in \eqref{LipX0}, i.e.,
\be
\|S(t) z_1-S(t)z_2\|_{\mathbb{X}_0}^2\leq C e^{Ct} \|z_1-z_2\|_{\mathbb{X}_0}^2.\non
\ee
Let us set $z_2=(\phi_{20}, \phi_{21})$ and $S(t)z_2=(\phi_2(t), \phi_{2t}(t))$.
Recalling that the initial datum is in $\mathcal{V}_1$, we have the uniform estimate (cf. \eqref{disa1a})
\be
\|S(t) z_2\|_{\mathbb{X}_1}\leq C(\|z_2\|_{\mathbb{X}_1}),\non
\ee
which also implies $\|\phi_{2tt}(t)\|_{-3}\leq C$.
Concerning the second term on the right-hand side of \eqref{holder}, we infer that
\bea
&& \|S(t) z_2-S(\tau)z_2\|_{\mathbb{X}_0}^2\non\\
&=& \|\phi_2(t)-\phi_2(\tau)\|_2^2+\|\phi_{2t}(t)-\phi_{2t}(\tau)\|_{-1}^2\non\\
&\leq& C\|\phi_2(t)-\phi_2(\tau)\|_3^\frac{4}{3}\|\phi_2(t)-\phi_2(\tau)\|^\frac23+C \|\phi_2(t)-\phi_2(\tau)\|^2\non\\
&&\quad   +C\| \overline{\phi_2}_t(t)-\overline{\phi_2}_t(\tau)\|^\frac43 \|\overline{\phi_{2t}}(t)-\overline{\phi_{2t}}(\tau)\|_{-3}^\frac23+|\langle \phi_{2t}(t)\rangle-\langle \phi_{2t}(\tau)\rangle|^2\non\\
&\leq& C\left(\int_\tau^t \|\phi_{2t}(s)\| ds\right)^\frac23+C \left(\int_\tau^t \|\overline{\phi_{2tt}}(s)\|_{-3}ds\right)^\frac23
\non\\
&&\quad +|\langle \phi_{21}\rangle|^2\left(e^{-\frac{t}{\beta}}-e^{-\frac{\tau}{\beta}}\right)^2\non\\
&\leq& C(t^*) |t-\tau|^\frac23.\non
\eea
As a consequence, from the above estimates and \eqref{holder} we conclude that
\be
\|S(t)z_1-S(\tau)z_2\|_{\mathbb{X}_0}\leq C(t^*)\Big(\|z_1-z_2\|_{\mathbb{X}_0}+ |t-\tau|^\frac13\Big),
\ee
where $C(t^*)$ is a constant depending on $t^*$, $\|z_1\|_{\mathbb{X}_1}$ and  $\|z_2\|_{\mathbb{X}_1}$. This ends the proof.
\end{proof}
Based on the asymptotic smoothing property (Lemma \ref{lmassm}) and the H\"{o}lder continuity of the semigroup $S(t)$ on $\mathcal{V}_1$ (Lemma \ref{lmholder}), from the abstract result \cite[Proposition 1]{EMZ} we deduce the following
\bp
\label{lm5}  There exists
a bounded set $\mathcal{M}_0\subset\mathcal{V}_1$, closed and
of finite fractal dimension in $\mathbb{X}_0$, positively
invariant for the semigroup $S(t)$, such that, for some $\gamma_0>0$ and
$K_0\geq0$, there holds
 \be
 \mathrm{dist}_{\mathbb{X}_0}(S(t)\mathcal{V}_1,\mathcal{M}_0)\leq K_0 e^{-\gamma_0 t}.\label{exp1}
 \ee
\ep

\textit{Step 3. Enlarging the basin of attraction}.
In what follows, we aim to show
that \eqref{exp1} actually holds for any bounded subset $\mathcal{B}
\subset \mathbb{X}_0^{M,M'}$ instead of the more regular set $\mathcal{V}_1$,
but with different constants $K_\mathcal{B}$ and $\gamma_\mathcal{B}$. In other words, we
have to prove that the basin of exponential attraction coincides with $\mathbb{X}_0^{M,M'}$ (recall \eqref{exp}). For this purpose, we recall the transitivity of exponential
attraction (cf. \cite[Theorem 5.1]{FGMZ}), that is,
\bl\label{trans}
Let $\mathbb{X}$ be a metric space with distance function denoted by ${\rm dist}$. $S(t)$ is a semigroup acting on $\mathbb{X}$ such that
$$ {\rm dist}(S(t)z_1, S(t)z_2)\leq C_0e^{K_0t}{\rm dist}(z_1, z_2) \quad \text{for some}\ C, K>0.$$
 We further assume that there exist three subsets $B_1, B_2, B_3$  in $\mathbb{X}$ such that
$$ {\rm dist}_{\mathbb{X}}(S(t)B_1,B_2)\leq C_1e^{-\alpha_1t}, \quad {\rm dist}_{\mathbb{X}}(S(t)B_2,B_3)\leq C_2e^{-\alpha_2t}.$$
Then we have
$${\rm dist}_{\mathbb{X}}(S(t)B_1,B_3)\leq C'e^{-\alpha't},$$
where $C'=C_0C_1+C_2$ and $\alpha'=\frac{\alpha_1\alpha_2}{K_0+\alpha_1+\alpha_2}$.
\el

\textbf{Proof of Theorem \ref{exat}}.
Consider any bounded set $\mathcal{B}\subset \mathbb{X}_0^{M,M'}$ with radius given by
$R =\sup_{(\phi_0,\phi_1)\in \mathcal{B}}\|(\phi_0,\phi_1)\|_{\mathbb{X}_0}$.
For any $z_{01}=\left(\phi_{10},\phi_{11}\right)$, $z_{02}=\left(\phi_{20},\phi_{21}\right)
 \in\mathcal{B}$, by the Lipschitz continuity \eqref{LipX0}, we have
 \be
 \|S(t)z_{01}-S(t)z_{02}\|_{\mathbb{X}_0}
 \leq L_1^\frac12e^{\frac{1}{2}L_2 t}\|z_{01}-z_{02}\|_{\mathbb{X}_0}.\label{conss}
 \ee
 Besides, it follows from Proposition \ref{abs2} that
 \be
\mathrm{dist}_{\mathbb{X}_0}(S(t)\mathcal{B},\mathcal{V}_1)\leq M(R)e^{-\gamma t}.\label{lll}
 \ee
 Then we conclude from
\eqref{exp1}--\eqref{lll}, Proposition \ref{lm5} and  Lemma \ref{trans} that
\be
\mathrm{dist}_{\mathbb{X}_0}(S(t)\mathcal{B},\mathcal{M}_0)\leq K_\mathcal{B}e^{-\gamma_\mathcal{B} t},\non
 \ee
where
$$
 K_\mathcal{B}=L_1^\frac12M(R)+K_0,\quad\gamma_\mathcal{B}=\frac{\gamma
\gamma_0}{\frac12 L_2+\gamma+\gamma_0}.
$$
Therefore, the set $\mathcal{M}_0$ has $\mathbb{X}_0^{M,M'}$ as basin of attraction. The proof of Theorem \ref{exat} is finished.
$\square$

\section{Convergence to equilibria}
\setcounter{equation}{0}
\noindent
In this section, we investigate the longtime behavior of a single trajectory $(\phi, \phi_t)$. More precisely, we show that each (energy) solution does converge to a single equilibrium. The main result is as follows

\bt \label{convergence} For any initial
datum $(\phi_0, \phi_1) \in \mathbb{X}_0$, the unique (energy) solution $\phi$ to problem
\eqref{e1}--\eqref{e2}   fulfills
\be
\lim_{t\rightarrow +\infty}\|\phi(\cdot,
t)-\phi_\infty\|_{2}+\|\phi_t(t)\|_{-1}=0. \label{conv1}
 \ee
 Here, $\phi_\infty$ is
a stationary solution to problem \eqref{e1}--\eqref{e2}, i.e., a solution
to the following elliptic equation subject to periodic boundary conditions with an average constraint:
\be \left\{\begin{array}{l}  \Delta^2 \phi_\infty +2\Delta \phi_\infty + f(\phi_\infty)=Const., \;\;\;\; x \in \mathbb{T}^n,\\
\langle\phi_\infty\rangle=M=\beta\langle\phi_1\rangle+\langle\phi_0\rangle.
  \end{array}
 \label{sta}
 \right.
 \ee
 Moreover, the following convergence rate estimates hold
 \bea
 && |\langle\phi(t)\rangle- M|\leq \beta |\langle\phi_1\rangle| e^{-\frac{t}{\beta}},\quad  |\langle\phi_t(t)\rangle|\leq |\langle\phi_1\rangle| e^{-\frac{t}{\beta}}, \label{rate1}\\
 && \|\overline{\phi}(t)-\overline{\phi_\infty}\|_{2}+\|\overline{\phi_t}(t)\|_{-1}\leq  C(1+t)^{-\frac{\theta}{1-2\theta}}, \label{rate2}
 \eea
for all $t\geq 0$, where $C$ is a constant depending on $\|(\phi_0, \phi_1)\|_{\mathbb{X}_0}$ and on the coefficients of the system, while $\theta\in (0, \frac12)$ may depend on $\phi_\infty$.
 \et

 First, we show the decay property of $\phi_t$, the time derivative of the phase-field.

 \begin{proposition}\label{phit}
Let the assumptions of Theorem \ref{convergence} hold. Then we have
 \be
 \lim_{t\to+\infty} \|{\phi}_t(t)\|_{-1}=0.\label{decaypt}
 \ee
 \end{proposition}
 \begin{proof}
 It follows from \eqref{mde1} that
 \be
 \lim_{t\to+\infty} \langle\phi_t(t)\rangle=0.\label{decaympt}
 \ee
 Recalling the energy equality \eqref{enereq}, using the uniform estimate \eqref{disa1} and the Sobolev embedding theorem, we have for $t\geq s\geq 0$
 \be
 \mathcal{E}(t)-\mathcal{E}(s)+\int_{s}^{t} \|\overline{\phi_t}(\tau)\|_{-1}^2d\tau \leq C \int_s^t e^{-\frac{\tau}{\beta}}d\tau.\label{endiff}
 \ee
 As we have seen before, $\mathcal{E}(t)$ is bounded from below by a constant (depending on $Q$). Thus \eqref{endiff} yields
 \be
 \int_{0}^{+\infty} \|\overline{\phi_t}(t)\|_{-1}^2dt <+\infty. \label{intphit}
 \ee
 On the other hand, since $\phi_{tt}\in L^\infty(0, +\infty; H^{-4}_p(Q))$, we have that
 \be
 \int_{t}^{t+1} \|\overline{\phi_{tt}}(s)\|_{-4} ds\leq C, \quad \ \forall \, t\geq 0.
 \ee
 As a result, the function $v(t):=\overline{\phi_t}$ is uniformly Lipschitz continuous in $H^{-4}_p(Q)$. This and \eqref{intphit} imply that $\|\overline{{\phi}_t}(t)\|_{-4} \to 0$ as $t\to +\infty$. Since the trajectory is precompact in $\mathbb{X}_0$ (recalling Section 5), we then have $\|\overline{{\phi}_t}(t)\|_{-1} \to 0$ as $t\to +\infty$. Together with \eqref{decaympt}, we arrive at \eqref{decaypt}. The proof is complete.
 \end{proof}

  Thanks to Proposition \ref{phit} we can define the $\omega$-limit set of $(\phi_0, \phi_1)$ as follows
  $$ \omega(\phi_0, \phi_1)=\{(\phi_\infty, 0): \phi_\infty \in H^2_p(Q), \ \exists \ t_n\nearrow +\infty, \ \|\phi(t_n)-\phi_\infty\|_{2}\to 0\}.$$
  Next, we give a characterization of the $\omega$-limit set. For any $M\in \mathbb{R}$, we set
  $\mathfrak{S}_M=\{\psi:  \psi \ \mbox{satisfies problem}\ \eqref{sta}\}. $
  It is standard to show that the energy functional
$E(\phi)$ admits at least one minimizer $\psi\in H^2_p(Q)$ with $\langle \psi\rangle=M$, which solves \eqref{sta}. As a consequence, the set $\mathfrak{S}_M$ is nonempty. Moreover, by standard elliptic estimate combined with a bootstrap argument, we see that the solution to problem \eqref{sta} is indeed smooth. Then we have

 \begin{proposition} Let the assumptions of Theorem \ref{convergence} hold. The $\omega$-limit set of $(\phi_0, \phi_1)$ is nonempty and is given by
 $$ \omega(\phi_0, \phi_1)=\{(\phi_\infty, 0): \phi_\infty\in \mathfrak{S}_M, \ \text{with}\ M=\beta\langle\phi_1\rangle+\langle\phi_0\rangle\}.$$
 Moreover, the energy functional $E(\phi)$ is constant on $\omega(\phi_0, \phi_1)$.
 \end{proposition}
 \begin{proof}
 For any $M\in \mathbb{R}$, we introduce the auxiliary functions
 \be
 f_M(y)=f(y+M)\quad \mbox{and}\quad F_M(y)=F(y+M).
 \ee
 Setting $M=\beta\langle\phi_1\rangle+\langle\phi_0\rangle$, we note that, for any solution $\phi$ to problem \eqref{e1}--\eqref{e2}, the following relation holds
\be
f(\phi)=f_M\big(\overline{\phi}-\beta \langle\phi_1\rangle e^{-\frac{t}{\beta}}\big).
\ee
 Then we rewrite equation \eqref{e1} in the following form
 \bea
&& \beta \overline{\phi_{tt}}+\overline{\phi_t}+A_0 (\Delta ^2 \overline{\phi} +2\Delta \overline{\phi}+f_M(\overline{\phi})-\langle f_M(\overline{\phi})\rangle)\non\\
&=& A_0(f_M(\overline{\phi})-\langle f_M(\overline{\phi})\rangle-f(\phi)+\langle f(\phi)\rangle). \label{e1v}
\eea
Testing \eqref{e1v} by $A^{-1}\overline{\phi_t}$, we obtain
 \bea
 && \frac{d}{dt}\left(\frac{\beta}{2}\|\overline{\phi_t}\|_{-1}^2+\frac12\|\Delta \overline{\phi}\|^2-\|\nabla \overline{\phi}\|^2+\int_Q F_M(\overline{\phi})dx\right)+\|\overline{\phi_t}\|_{-1}^2\non\\
 &=&(f_M(\overline{\phi})-f(\phi), \overline{\phi_t}).\label{Ly1}
 \eea
 Using the uniform estimate \eqref{disa1}, the right-hand side of \eqref{Ly1} can be estimated as follows
 \bea
 &&(f_M(\overline{\phi})-f(\phi), \overline{\phi_t})\non\\
 &=& \left(f_M(\overline{\phi})-f_M\big(\overline{\phi}-\beta \langle\phi_1\rangle e^{-\frac{t}{\beta}}\big), \overline{\phi_t}\right)\non\\
 &=& \beta\langle\phi_1\rangle e^{-\frac{t}{\beta}} \left( \int_0^1 f'_M(s\overline{\phi}+(1-s)\big(\overline{\phi}-\beta \langle\phi_1\rangle e^{-\frac{t}{\beta}}\big)) ds, \overline{\phi_t}\right)\non\\
 &\leq& |\beta\langle\phi_1\rangle| e^{-\frac{t}{\beta}} \left\| \int_0^1 f'_M(s\overline{\phi}+(1-s)\big(\overline{\phi}-\beta \langle\phi_1\rangle e^{-\frac{t}{\beta}}\big)) ds\right\|_{1}\| \overline{\phi_t}\|_{-1}\non\\
 &\leq& \frac14  \|\overline{\phi_t}\|_{-1}^2+C  e^{-\frac{2t}{\beta}}.\non
 \eea
 As a result, we have
\be
 \frac{d}{dt}\left(\frac{\beta}{2}\|\overline{\phi_t}\|_{-1}^2+\frac12\|\Delta \overline{\phi}\|^2-\|\nabla \overline{\phi}\|^2+\int_Q F_M(\overline{\phi})dx\right)+\frac34 \|\overline{\phi_t}\|_{-1}^2\leq C  e^{-\frac{2t}{\beta}}.\label{Ly1a}
 \ee
Define
 \be
 \mathcal{G}(t)=\left(A_0^{-1}\overline{\phi_t}, A_0^{-1}\big(\Delta^2 \overline{\phi}+2\Delta \overline{\phi}+f_M(\overline{\phi})-\langle f_M(\overline{\phi})\rangle\big)\right)_{-1}. \label{G}
 \ee
 Thanks to Proposition \ref{phit} and uniform estimate \eqref{disa1}, we see that
 \be
 \lim_{t\to+\infty} \mathcal{G}(t)=0.
 \ee
 Differentiating $\mathcal{G}$ with respect to time and recalling \eqref{e1v}, we get
 \bea
 \frac{d}{dt}\mathcal{G}(t)
 &=& \left(A_0^{-1}\overline{\phi_{tt}}, A_0^{-1}\big(\Delta^2 \overline{\phi}+2\Delta \overline{\phi}+f_M(\overline{\phi})-\langle f_M(\overline{\phi})\rangle\big)\right)_{-1}\non\\
 &&+ \left(A_0^{-1}\overline{\phi_t}, A_0^{-1}\big(\Delta^2 \overline{\phi_t}+2\Delta \overline{\phi_t}+f_M'(\overline{\phi})\overline{\phi_t}-\langle f'_M(\overline{\phi}) \overline{\phi_t}\rangle\big)\right)_{-1}\non\\
 &=& -\frac{1}{\beta} \left(A_0^{-1}\overline{\phi_{t}}, A_0^{-1}\big(\Delta^2 \overline{\phi}+2\Delta \overline{\phi}+f_M(\overline{\phi})-\langle f_M(\overline{\phi})\rangle\big)\right)_{-1}\non\\
 && -\frac{1}{\beta}\|\Delta \overline{\phi}+2\Delta \overline{\phi}+f_M(\overline{\phi})-\langle f_M(\overline{\phi})\rangle\|_{-2}^2\non\\
 && +\frac{1}{\beta} \left(f_M(\overline{\phi})-f(\phi), A_0^{-1}\big(\Delta^2 \overline{\phi}+2\Delta \overline{\phi}+f_M(\overline{\phi})-\langle f_M(\overline{\phi})\rangle\big)\right)_{-1}\non\\
 &&+ \|\overline{\phi_t}\|_{-1}^2 -2\|\overline{\phi_t}\|_{-2}^2 + \left(A_0^{-1}\overline{\phi_t}, A_0^{-1}\big(f_M'(\overline{\phi})\overline{\phi_t}-\langle f'_M(\overline{\phi}) \overline{\phi_t}\rangle\big)\right)_{-1}.
 \eea
 Using the uniform estimate \eqref{disa1}, the  H\"{o}lder inequality and Young's inequality, we deduce that
 \be \frac{d}{dt}\mathcal{G}(t)+ \frac{1}{2\beta}\|\Delta \overline{\phi}+2\Delta \overline{\phi}+f_M(\overline{\phi})-\langle f_M(\overline{\phi})\rangle\|_{-2}^2\leq C_1 \|\overline{\phi_t}\|_{-1}^2+ Ce^{-\frac{2t}{\beta}}.\label{dG}
 \ee
 Let us introduce the function
 \be
 \mathcal{W}(t)=\beta\|\overline{\phi_t}\|_{-1}^2+\|\Delta \overline{\phi}\|^2-2\|\nabla \overline{\phi}\|^2+2\int_Q F_M(\overline{\phi})dx+\nu \mathcal{G}(t). \label{Ly3}
 \ee
 where $\nu>0$ is sufficiently small so that $C_1\nu\leq \frac12$.

 From the above estimates \eqref{Ly1a} and \eqref{dG}, it follows that
 \be
 \frac{d}{dt} \mathcal{W}(t)+ \|\overline{\phi_t}\|_{-1}^2+\frac{1}{\beta}\|\Delta \overline{\phi}+2\Delta \overline{\phi}+f_M(\overline{\phi})-\langle f_M(\overline{\phi})\rangle\|_{-2}^2\leq Ce^{-\frac{2t}{\beta}}, \label{Ly4}
 \ee
 where the term on the right-hand side $e^{-\frac{2t}{\beta}}$ is integrable on $[0,+\infty)$ and $$\lim_{t\to +\infty} \int_t^\infty e^{-\frac{2s}{\beta}} ds=0.$$
  Similarly to \eqref{endiff}, we have
 \be
 \mathcal{W}(t)-\mathcal{W}(s)\leq C\int_s^t e^{-\frac{2\tau}{\beta}}d\tau,\label{ineqW}
 \ee
 for $0\leq s\leq t<+\infty$. This yields that there exists $\mathcal{W}_\infty\in \mathbb{R}$ such that
 \be
 \lim_{t\to+\infty} \mathcal{W}(t)=\mathcal{W}_\infty.\label{conW}
 \ee

 Since the trajectory is precompact in $\mathbb{X}_0$, we can find an unbounded increasing sequence $\{t_n\}_{n=1}^\infty$ in $\mathbb{R}^+$ such that $\|\phi(t_n)-\phi_\infty\|_{2}\to 0$ as $n\to +\infty$, for some $\phi_\infty\in H^2_p(Q)$. Thus the set $\omega(\phi_0, \phi_1)$ is nonempty.

We now show that any possible limit point $\phi_\infty$ belongs to the set $\mathfrak{S}_M$. First, it easily follows from \eqref{mde2} that $\langle\phi_\infty\rangle=M$.
 Without loss of generality, we may assume $t_{n+1}\geq t_n+1$ for $n\in \mathbb{N}$. Integrating \eqref{Ly4} with respect to time on the interval $[t_n, t_{n+1}]$, we infer from \eqref{conW} that
 \bea
 && \int_0^1 \|\overline{\phi_t}(t_n+t)\|_{-1}^2dt\non\\
 && +\frac{1}{\beta}\int_0^1 \|\Delta^2 \overline{\phi}(t_n+t)+2\Delta \overline{\phi}(t_n+t)+f_M(\overline{\phi}(t_n+t))-\langle f_M(\overline{\phi}(t_n+t))\rangle\|_{-2}^2 dt\non\\
 &\leq& \int_{t_n}^{t_{n+1}} \|\overline{\phi_t}(s)\|_{-1}^2+\frac{1}{\beta}\|\Delta \overline{\phi}(s)+2\Delta \overline{\phi}(s)+f_M(\overline{\phi}(s))-\langle f_M(\overline{\phi}(s))\rangle\|_{-2}^2 ds\non\\
 &\leq&  -\mathcal{W}(t_{n+1})+\mathcal{W}(t_n)+C\int_{t_n}^{t_{n+1}}  e^{-\frac{2s}{\beta}}ds\non\\
 &\to& 0, \quad \mbox{as} \ n\to +\infty.\label{ccc}
 \eea
 Then we have $\|\overline{\phi}(t_n+t_1)-\overline{\phi}(t_n+t_2)\|_{-1}\to 0$ uniformly for $t_1, t_2\in [0,1]$ as $n\to +\infty$. Thus, from the definition of $t_n$, we see that for $t\in [0,1]$, $\|\overline{\phi}(t_n+t)-\phi_\infty\|_{2}\to 0$. By the Lebesgue dominated convergence theorem, we deduce from \eqref{ccc} that
 \be
 \|\Delta^2 \overline{\phi_\infty}+2\Delta \overline{\phi_\infty}+f_M(\overline{\phi_\infty})-\langle f_M(\overline{\phi_\infty})\rangle\|_{-2}=0.
 \ee
 Recalling the definition of $f_M$ and the fact $\langle\phi_\infty\rangle=M$, we easily see that $\phi_\infty$ solves the stationary problem \eqref{sta}.

  It follows from \eqref{endiff} that $\mathcal{E}(t)$ converges to a certain constant $E_\infty$ as $t\to+\infty$. Since we have shown the convergence of $\phi_t$ (cf. \eqref{decaypt}), we see that
 \be
 \lim_{t\to+\infty} E(\phi(t))=E_\infty.\label{limE}
 \ee
 As a consequence, $E(\phi)$ is constant on $\omega(\phi_0, \phi_1)$. The proof is complete.
\end{proof}

We are now able to prove the convergence of $\phi$ to a single equilibrium $\phi_\infty$. This is not a trivial issue since
the energy functional $E(\phi)$ is in general nonconvex (for instance, when $\epsilon>0$ is large). Therefore we  do not expect uniqueness of solutions for the stationary problem \eqref{sta}. More precisely, the $\omega$-limit set is a subset of $\mathfrak{S}_M$, whose structure might be complicated (e.g. it can be a continuum), we do not know whether the phase-field $\phi$ will converge or not as time goes to infinity, although the sequential convergence holds due to the precompactness of the trajectory.

To overcome this difficulty, which is typical of pattern formation models (cf. e.g., \cite{AW,AFI,FIP1,FIP2,FSc,GPS,GW,RH,WGZ07,WWu,ZWH} and references therein), we shall make use the well-known \L ojasiewicz--Simon approach (see, for instance, \cite{H06}).
 For any $M\in \mathbb{R}$, we consider the functional
 \be
 E_M(v)=\int_{Q} \left(\frac12|\Delta v|^2-|\nabla v|^2+F_M(v)\right) dx, \quad \forall\, v\in \dot{H}^2_p(Q).
 \ee
 \br\label{brE}
 It is obvious that $E_M(v)=E(v+M)$  for any $v\in \dot{H}^2_p(Q)$. Moreover, if $v$ is a critical point of $E_M$ in $\dot H^2_p(Q)$, then $\psi=v+M$ is a critical point of $E$ in the space $\{\phi\in H^2_p(Q): \langle\phi\rangle=M\}$ and vice versa.
 \er

 Then we establish a convenient \L ojasiewicz--Simon type inequality, namely,
 \bl\label{ls}
 Let $v^*$ be a critical point of $E_M(v)$ in $\dot H^2_p(Q)$. Then there exist constants
 $\theta\in(0,\frac12)$ and $\delta>0$ depending on $v^*$ such that,
 for any $v\in \dot{H}^2_p(Q)$ satisfying $\|v-v^*\|_{2}<\delta$,
 there holds
 \be
 \|\Delta^2 v+2\Delta v+f_M(v)-\langle f_M(v)\rangle\|_{-2} \geq
 |E_M(v)-E_M(v^*)|^{1-\theta}.\label{LSE}
 \ee
 \el
 \begin{proof}
Our hypotheses entail that $E_M(v)\in C^2(\dot{H}^2_p(Q); \mathbb{R})$. Observe that, for any $v,u\in \dot{H}^2_p(Q)$, we have
\bea
&& <E_M'(v), u>_{H^{-2}_p(Q), H^2_p(Q)}\non\\
&=&\int_Q \left[\Delta v\Delta u-2\nabla v\cdot \nabla u+ f_M(v) u\right] dx\non\\
&=&\int_Q \left[\Delta v\Delta u-2\nabla v\cdot \nabla u+( f_M(v)- \langle f_M(v)\rangle) u\right] dx.\non
\eea
 Then it is easy to check that any solution $v^*$ to the stationary problem \eqref{sta} is a critical point of the energy functional $E_M(v)$ in $\dot{H}^2_p(Q)$ such that $E_M'(v^*)=0$, and conversely, any
 critical point of $E_M(v)$ is a solution to \eqref{sta}. Let
 $$ E_M'(v)|_{\dot{H}^4_p(Q)}:=\mathcal{M}(v)= \Delta ^2 v-2\Delta v+f_M(v)-\langle f_M(v)\rangle: \dot{H}^4_p(Q)\to \dot{L}^2_p(Q).$$
 Thanks to the Sobolev embedding $H^2_p(Q)\hookrightarrow L^\infty_p( Q)$ $(n\leq3)$, we have that $\mathcal{M}(v)\in C^1(\dot{H}^4_p(Q); \dot{L}^2_p(Q))$ is analytic (cf. \cite{MR}). For any $u,v,w\in \dot{H}^4_p(Q)$, a direct calculation yields
 \be
 (\mathcal{M}'(w)v, u)=\int_Q \left[\Delta v\Delta u-2\nabla v\cdot \nabla u+(f'_M(w)v- \langle f'_M(w)v\rangle) u\right] dx.\non
 \ee
Observe now that, for any $w \in \dot{H}^4_p(Q)$, $\mathcal{L}(w)=\mathcal{M}'(w)$ is a bounded linear self-adjoint operator from $\dot{H}^4_p(Q)$ to $\dot{L}^2_p(Q)$. The leading order term of the linear operator $\mathcal{L}(w)$ is $\Delta^2:\dot H^4_p(Q) \to \dot L^2_p(Q)$ and its corresponding symmetric bilinear form is given by
 $$a(f, g)=\int_Q \Delta f\Delta g dx, \quad \forall\,f, g \in \dot{H}^2_p(Q).$$
   The remaining part of $\mathcal{L}(w)$ is a compact operator from $\dot H^4_p(Q)$ to $\dot L^2_p(Q)$. As a consequence, for any $w\in \dot H^4_p(Q)$, $\mathcal{L}(w)$ is indeed  a compact perturbation of a Fredholm operator of index zero from $H^4_p(Q)$ to $L^2_p(Q)$. We note that $\mathcal{L}=E''|_{\dot{H}^4_p(Q)}$. Then, for any $u,v,w\in \dot{H}^2_p(Q)$, we have
 \be
 <E''_M(w)v, u>_{H^{-2}_p(Q), H^2_p(Q)}=\int_Q \left[\Delta v\Delta u-2\nabla v\cdot \nabla u+(f'_M(w)v- \langle f'_M(w)v\rangle) u\right] dx.\non
 \ee
   For any critical point $v^*$, it follows that ${\rm Ker} E''(v^*) \subset H^4_p(Q)$
 and its range is closed in $\dot{L}^2_p(Q)$ and $(\dot{H}^2_p(Q))^*$, respectively, so that $\dot{L}^2_p(Q)={\rm Ker} E''(v^*)\oplus {\rm Ran} \mathcal{L}(v^*)$, $(\dot{H}^2_p(Q))^*={\rm Ker} E''(v^*)\oplus {\rm Ran} E''(v^*)$.
 Here, $(\dot{H}^2_p(Q))^*$ is the dual space of $\dot{H}^2_p(Q)$, which is the space of classes
 $$ [f] =\{ f +g;\, g \in H^{-2}_p(Q),\,< g, h>_{H^{-2}_p(Q), H^{2}_p(Q)}=0, \ \forall\, h \in \dot H^{2}_p(Q)\},$$
 endowed with the norm $\|[f]\|_{(\dot{H}^2_p(Q))^*}:=\|f-\langle f\rangle\|_{-2}$. Therefore we are in a position to apply the abstract result \cite[Corollary 3.11]{Chill} to conclude that there exist constants
 $\theta\in(0,\frac12)$ and $\delta>0$ depending on $v^*$ such that,
 for any $v\in \dot{H}^2_p(Q)$ satisfying $\|v-v^*\|_{2}<\delta$, there holds
  \be
 \|E'_M(v)\|_{(\dot{H}^2_p(Q)^*} \geq
 |E_M(v)-E_M(v^*)|^{1-\theta},\non
 \ee
 which yields \eqref{LSE}. The proof is complete.
 \end{proof}

 For every $(\phi_\infty, 0)\in \omega(\phi_0, \phi_1)$, we set $v^*=\phi_\infty-M$, then $\langle v^*\rangle=0$. By Lemma \ref{ls}, there exist some $\delta$ and $\theta\in (0, \frac12)$ that may depend on $v^*$ such that the inequality \eqref{LSE} holds for $v\in \mathbf{B}_{\delta}(v^*):=\{v\in \dot{H}^2_p(Q): \|v-v^*\|_{2}<\delta\}$ and $|E_M(v)-E_M(v^*)|\leq 1$. The union of balls $\{ \mathbf{B}_{\delta}(v^*): (v^*+M,0)\in \omega(\phi_0, \phi_1)\}$ forms an open covering of $\omega(\phi_0, \phi_1)$. Due to the compactness of $\omega(\phi_0, \phi_1)$ in $\mathbb{X}_0$, we can find a finite sub-covering $\{\mathbf{B}_{\delta_i}(v^*_i)\}_{i=1,2,...,m}$, where the constants $\delta_i, \delta_i$ corresponding to $v^*_i$ in Lemma \ref{ls} are indexed by $i$.

 From the definition of $\omega(\phi_0, \phi_1)$, we know that there exists a sufficiently large $t_0$ such that $\overline{\phi}(t)\in \mathcal{U}:=\bigcup_{i=1}^m\mathbf{B}_{\delta_i}(v^*_i)$ for $t\geq t_0$. Taking $ \theta=\min_{i=1}^m\{\theta_i\}\in (0, \frac12)$, we infer from \eqref{limE}, Remark \ref{brE} and Lemma \ref{ls} that, for all $t\geq t_0$,
 \be
 \|\Delta^2 \overline{\phi}+2\Delta \overline{\phi}+f_M(\overline{\phi})-\langle f_M(\overline{\phi})\rangle\|_{-2}\geq
 |E_M(\overline{\phi}(t))-E_\infty|^{1-\theta}.\label{LSEa}
 \ee
 Let us now set
  \be Z(t)=\left(\|\overline{\phi_t}\|_{-1}^2+\frac{1}{\beta}\|\Delta \overline{\phi}+2\Delta \overline{\phi}+f_M(\overline{\phi})-\langle f_M(\overline{\phi})\rangle\|_{-2}^2\right)^\frac12+ e^{-\frac{t}{\beta}}.\label{ZZ}
  \ee
  From \eqref{Ly4} and \eqref{conW} we infer that
  \be
 \int_t^{+\infty} Z(\tau)^2 d\tau\leq \mathcal{W}(t)-\mathcal{W}_\infty + C e^{-\frac{2t}{\beta}}.\label{LSEc}
 \ee
On the other hand, using the \L ojasiewicz--Simon inequality \eqref{LSEa}, the uniform estimates \eqref{disa1} and the fact $\frac{1}{1-\theta}<2$, we deduce that, for all $t\geq t_0$,
\bea
|\mathcal{W}(t)-\mathcal{W}_\infty|&\leq &
\beta \|\overline{\phi_t}\|_{-1}^2+ |E_M(\overline{\phi})-E_\infty| + \nu \mathcal{G}(t)\non\\
&\leq& \|\overline{\phi_t}\|_{-1}^2+ \|\Delta^2 \overline{\phi}+2\Delta \overline{\phi}+f_M(\overline{\phi})-\langle f_M(\overline{\phi})\rangle\|_{-2} ^{\frac{1}{1-\theta}}\non\\
&& +C\|\overline{\phi_t}\|_{-1}\|\Delta^2 \overline{\phi}+2\Delta \overline{\phi}+f_M(\overline{\phi})-\langle f_M(\overline{\phi})\rangle\|_{-2}\non\\
&\leq& C\|\overline{\phi_t}\|_{-1}^{\frac{1}{1-\theta}} + \|\Delta^2 \overline{\phi}+2\Delta \overline{\phi}+f_M(\overline{\phi})-\langle f_M(\overline{\phi})\rangle\|_{-2} ^{\frac{1}{1-\theta}}.\label{LSEb}
\eea
This gives
\be
\int_t^{+\infty}Z(\tau)^2d\tau \leq  CZ(t)^\frac{1}{1-\theta}, \quad \forall\, t\geq t_0.\label{Z}
\ee
Recall now the following result (cf. \cite[Lemma 7.1]{FS}, see also \cite[Lemma 4.1]{HT01})
\bl\label{AEle}
Let $\theta\in (0,\frac12)$. Assume that $Z\geq 0$ is a measurable function on $(0,+\infty)$ such that $Z\in L^2(\mathbb{R}^+)$ and suppose that there exist $C>0$ and $t_0\geq 0$ such that
 \be
\int_t^{+\infty}Z(\tau)^2 d\tau\leq C Z(t)^\frac{1}{1-\theta},\quad \text{for a.a.}\ \  t\geq t_0.\non
\ee
 Then $Z\in L^1(t_0, +\infty)$.
\el
As a consequence, we infer from \eqref{Z}, the definition of $Z$ (cf. \eqref{ZZ}) and Lemma \ref{AEle}  that
\be
\int_0^{+\infty} \|\overline{\phi_t}(t)\|_{-1} d t<+\infty. \label{intt}
\ee
This entails the convergence of $\overline{\phi}(t)$ in $H^{-1}_p(Q)$. Due to the precompactness of the trajectory in $\mathbb{X}_0$, we see that there exists a steady state $\phi_\infty\in \mathfrak{S}_M$ such that
\be
\lim_{t\to+\infty} \|\overline{\phi}(t)-\overline{\phi_\infty}\|_2=0.
\ee
Since we already know the convergences of the mean value of $\phi$ (cf. \eqref{mde2}) and $\phi_t$ (cf. Proposition \ref{phit}), we conclude that
$ \omega(\phi_0, \phi_1)=(\phi_\infty, 0)$ and \eqref{conv1} holds.

Finally, it remains to prove estimates \eqref{rate1} and \eqref{rate2} on the convergence rate. The former is a direct consequence of  \eqref{mde1} and \eqref{mde2}. Concerning the latter, observe that inequality \eqref{Ly4} implies that (cf. also \eqref{ineqW})
$$ \widetilde{\mathcal{W}}(t)=\mathcal{W}(t)+\frac{C\beta}{2}e^{-\frac{2t}{\beta}}-\mathcal{W}_\infty \geq 0 $$ fulfills
\be
 \frac{d}{dt} \widetilde{\mathcal{W}}(t)+ \|\overline{\phi_t}\|_{-1}^2+\frac{1}{\beta}\|\Delta \overline{\phi}+2\Delta \overline{\phi}+f_M(\overline{\phi})-\langle f_M(\overline{\phi})\rangle\|_{-2}^2\leq 0, \label{Ly4a}
 \ee
  Thus, $\widetilde{\mathcal{W}}(t)$ is decreasing in time and  recalling \eqref{conW} we have
  $\lim_{t\to+\infty} \widetilde{\mathcal{W}}(t)=0$.
 Moreover, for $t\geq t_0$, we infer from \eqref{LSEb} and \eqref{Ly4a} that
 \bea
 \widetilde{\mathcal{W}}(t)^{2(1-\theta)}&\leq& C( \|\overline{\phi_t}\|_{-1}^2+\frac{1}{\beta}\|\Delta \overline{\phi}+2\Delta \overline{\phi}+f_M(\overline{\phi})-\langle f_M(\overline{\phi})\rangle\|_{-2}^2)+ Ce^{-\frac{2t}{\beta}}\non\\
 &\leq& -C\frac{d}{dt} \widetilde{\mathcal{W}}(t)+ Ce^{-\frac{2t}{\beta}},
 \eea
which yields  the decay rate of $\widetilde{\mathcal{W}}(t)$, namely,
 \be
 0\leq \widetilde{\mathcal{W}}(t)\leq C(1+t)^{-\frac{1}{1-2\theta}}, \quad \forall\, t\geq 0.\non
 \ee
 Thus it follows from \eqref{Ly4a} that, for any $t\geq t_0$,
 \bea
  &&\int_t^{2t}  \|\overline{\phi_t}\|_{-1} d\tau\non\\
  &\leq& t^{\frac12}\left(\int_t^{2t} ( \|\overline{\phi_t}\|_{-1}^2+\frac{1}{\beta}\|\Delta \overline{\phi}+2\Delta \overline{\phi}+f_M(\overline{\phi})-\langle f_M(\overline{\phi})\rangle\|_{-2}^2) d\tau\right)^\frac12\non\\
   &\leq&  Ct^\frac12 \widetilde{\mathcal{W}}(t)^\frac12\non\\
   &\leq&  C(1+t)^{-\frac{\theta}{1-2\theta}}.\non
 \eea
 As a consequence, we have (recall \eqref{intt})
 \bea
 && \int_t^{+\infty} \|\overline{\phi_t}\|_{-1} d\tau= \sum_{j=0}^{+\infty} \int_{2^j t}^{2^{j+1}t} \|\overline{\phi_t}\|_{-1} d\tau\non\\
 &\leq&
  C\sum_{j=0}^{+\infty} (2^j t)^{-\frac{\theta}{1-2\theta}}\leq  C(1+t)^{-\frac{\theta}{1-2\theta}}, \quad \forall\, t\geq t_0,\non
 \eea
 which gives
 \be
 \|\overline{\phi}(t)-\overline{\phi_\infty}\|_{-1}\leq  C(1+t)^{-\frac{\theta}{1-2\theta}}, \quad \forall\, t\geq 0. \label{ratept}
 \ee
 In order to obtain the decay rate in $\mathbb{X}_0$ norm, we test \eqref{e1a} by $A_0^{-1} \overline{\phi_t}$, $A_0^{-1} (\overline{\phi}-\overline{\phi_\infty})$, respectively. We thus obtain
\bea
&&\frac{d}{dt}\left(\frac{\beta}{2}\|\overline{\phi_t}\|_{-1}^2+ \frac12\|\Delta (\phi-\phi_\infty)\|^2-\|\nabla (\phi-\phi_\infty)\|^2\right)+\|\overline{\phi_t}\|_{-1}^2\non\\
&=& -\int_Q(f(\phi)-f(\phi_\infty))\overline{\phi_t} dx\non\\
&\leq& \|f(\phi)-f(\phi_\infty)\|_1\|\overline{\phi_t}\|_{-1}\non\\
&\leq& \frac12\|\overline{\phi_t}\|_{-1}+C\|\phi-\phi_\infty\|_1^2, \label{dec1}
\eea
and
\bea
&& \frac{d}{dt}\left(\beta (\overline{\phi_t}, \overline\phi-\overline{\phi_\infty})_{-1}+\frac12\|\overline \phi-\overline{\phi_\infty}\|_{-1}^2\right)-\beta \|\overline{\phi_t}\|^2_{-1}+\|\Delta (\phi-\phi_\infty)\|^2\non\\
&=&-\int_{Q} (f(\phi)-f(\phi_\infty))( \overline \phi-\overline{\phi_\infty}) dx+2\|\nabla (\phi-\phi_\infty)\|^2\non\\
&\leq& C\|\phi-\phi_\infty\|^2_1.\label{dec2}
\eea
Multiplying \eqref{dec2} by a sufficiently small constant $\eta>0$, adding the resulting inequality to \eqref{dec1}, using interpolation and Young's inequality, we arrive at
\bea
\frac{d}{dt} \mathcal{Y}(t)+C\mathcal{Y}(t)&\leq& C\|\phi-\phi_\infty\|^2_{-1}\non\\
&\leq& C\|\overline{\phi}-\overline{\phi_\infty}\|^2_{-1}+C|\langle \phi\rangle-\langle\phi_\infty\rangle|^2\non\\
&\leq& C(1+t)^{-\frac{2\theta}{1-2\theta}},\label{dyt}
\eea
where
\bea
\mathcal{Y}(t)&=& \frac{\beta}{2}\|\overline{\phi_t}\|_{-1}^2+ \frac12\|\Delta (\phi-\phi_\infty)\|^2-\|\nabla (\phi-\phi_\infty)\|^2+\eta \beta (\overline{\phi_t}, \overline\phi-\overline{\phi_\infty})_{-1}\non\\
&& +\frac\eta 2\|\overline \phi-\overline{\phi_\infty}\|_{-1}^2.\non
\eea
Note that, for properly large $C_1>0$, we have
\be
\mathcal{Y}(t)+C_1\|\overline \phi-\overline{\phi_\infty}\|_{-1}^2\geq C_2 (\|\overline{\phi}-\overline{\phi_\infty}\|_2^2+ \|\overline{\phi_t}\|_{-1}^2).\label{Y}
\ee
On the other hand, \eqref{dyt} gives
\be
\mathcal{Y}(t)\leq  C(1+t)^{-\frac{2\theta}{1-2\theta}}.\non
\ee
Combining \eqref{ratept} with \eqref{Y}, we conclude that \eqref{rate1} holds. The proof of Theorem \ref{convergence} is now complete. $\square$

\medskip

{\bf Acknowledgments.} Maurizio Grasselli gratefully acknowledges the support of Fudan University Key Laboratory for Contemporary Mathematics through the Senior Visiting Scholarship. Hao Wu was
partially supported by National Science Foundation of China 11001058, SRFDP and ``Chen Guang" project supported by Shanghai Municipal Education Commission and Shanghai Education Development Foundation.

\end{document}